\documentclass[leqno]{siamart1116}
\usepackage{amssymb}
\usepackage{graphicx}
\usepackage{subcaption}
\usepackage{xspace}
\usepackage{multirow}
\usepackage{siunitx}
\usepackage{bm,bbm}
\usepackage{mydef}
\usepackage[numbers]{natbib}
\let\cite=\citet
\usepackage{enumerate}
\usepackage{dsfont}

\begin{document}

\newcommand\footnotemarkfromtitle[1]{%
\renewcommand{\thefootnote}{\fnsymbol{footnote}}%
\footnotemark[#1]%
\renewcommand{\thefootnote}{\arabic{footnote}}}

\title{Positive asymptotic preserving approximation of the radiation transport equation\footnotemark[1]}

\author{Jean-Luc Guermond\footnotemark[2] \and Bojan Popov\footnotemark[2] 
\and Jean Ragusa\footnotemark[3]}

\date{Draft version \today}

\maketitle

\renewcommand{\thefootnote}{\fnsymbol{footnote}} \footnotetext[1]{This
  material is based upon work supported by a ``Computational R\&D
    in Support of Stockpile Stewardship'' grant from Lawrence
    Livermore National Laboratory, the National Science Foundation
  grants DMS-1619892 and DMS-1620058, by the Air Force Office of
  Scientific Research, USAF, under grant/contract number
  FA9550-15-1-0257, and by the Army Research Office under
  grant/contract number W911NF-15-1-0517. Draft version, \today }
\footnotetext[2]{Department of Mathematics, Texas A\&M University 3368
  TAMU, College Station, TX 77843, USA} \footnotetext[3]{Department of
  Nuclear engineering, College Station, TX 77843, USA}
\renewcommand{\thefootnote}{\arabic{footnote}}

\begin{abstract}
  We introduce a (linear) positive and asymptotic preserving method
  for solving the one-group radiation transport equation.  The
  approximation in space is discretization agnostic: the space
  approximation can be done with continuous or discontinuous finite
  elements (or finite volumes, or finite differences).  The method is
  first-order accurate in space. This type of accuracy is coherent
  with Godunov's theorem since the method is linear.  The two key
  theoretical results of the paper are
  Theorem~\ref{Th:diffusion_limit_AP} and
  Theorem~\ref{Th:AP_positivity}. The method is illustrated with
  continuous finite elements. It is observed to converge with the rate
  $\calO(h)$ in the $L^2$-norm on manufactured solutions, and it
  is $\calO(h^2)$ in the diffusion regime. Unlike other standard
  techniques, the proposed method does not suffer from overshoots at
  the interfaces of optically thin and optically thick regions
\end{abstract}



\pagestyle{myheadings} \thispagestyle{plain} \markboth{J.-L. GUERMOND, B. POPOV, 
J. RAGUSA}{Approximation of the radiation transport equation}

\section{Introduction}\label{sec:Introd}
Constructing approximations of the radiation transport equation that
are both positive and robust, \ie do not lock, in the diffusion
limit is a difficult task.  Diffusive and optically thick regimes
occur when the physical medium is many mean-free-path thick and the
interaction processes are dominated by scattering (\ie absorption is
weak or non existent). Here the words ``robust'' and ``locking'' are
used in the sense defined by \cite{Babuska_Suri_1992}; this
terminology is common in the elliptic literature.  In the
hyperbolic literature, approximation techniques that are robust with
parameters tending to  limiting values are often called
asymptotic-preserving in reference to \cite{Jin_1999}. These two
terminologies are use interchangeably in the paper.

In the wake of \cite{Reed_Hill_1973} and \cite{Lesaint_Raviart_1974},
a dominant paradigm in the kinetic literature to solve the radiation
transport equation consists of using the discontinuous Galerkin (dG)
technique with the upwind flux. Unfortunately, to the best of our
knowledge, there does not exist yet in the literature a discontinuous
Galerkin technique that is both positive and does not lock in the
thick diffusion limit.  For instance, it was pointed out in
\cite{Larsen_1983} that the finite volume scheme ``step scheme'' (\ie
piecewise constant dG) with standard upwind, locks in the diffusion
limit. Several variations of the ``step scheme'' have been analyzed in
\cite{Larsen_Morel_Miller_1987}: it was shown that the ``Lund-Wilson''
and the ``Castor'' variants yield cell-edge
angular fluxes that also lock in the diffusion limit.
Furthermore, the cell-edge fluxes for these schemes cannot reproduce
the infinite medium solution.  A ``new'' scheme was proposed in
\cite{Larsen_Morel_Miller_1987} but was subsequently dismissed due to
a poor behavior at the boundaries.  For many years, the
diamond-difference scheme was found to be the best performing
finite-difference scheme, even though its cell-edge fluxes lock in the
thick diffusion limit.  In \cite{Larsen_Morel_1989}, most of the
previous schemes have been set aside in favor of the linear
discontinuous finite element scheme (the piecewise linear dG technique
with standard upwinding).

The cause for locking has been identified in a seminal paper
by~\cite{Adams_2001}. The author analyzed multi-dimensional dG
approximations and showed that some dG schemes lock in the diffusion
limit because the upwind numerical flux forces the scalar flux, and thus the
angular flux, to be continuous across the mesh cells. This observation
has been confirmed in~\cite{Guermond_Kanschat_2010}, where the
equivalence of the limit problem to a mixed discretization for the
Laplacian was proved and the nature of the boundary layers was
discussed.  The asymptotic analysis in~\citep{Adams_2001}
and~\citep{Guermond_Kanschat_2010} suggests that the problem could be
alleviated by modifying the upwind numerical flux. 
By making the amount of stabilization dependent on the scattering
cross section so that the amount of upwinding decreases as the
scattering cross section increases, it is shown in
\cite{Ragusa_Guermond_Kanschat_2012} that locking can indeed be
avoided in the thick diffusive limit, including for the dG0
approximation.  The dG scheme thus obtained converges robustly for
finite element spaces of any polynomial order including piecewise
constant functions (dG0), but, like the other methods mentioned above,
it is not guaranteed to be positive.

The objective of this work is to revisit the approximation theory for
the radiation transport equation in heterogeneous media by using the
algebraic framework (\ie discretization-independent) introduced in
\cite{Guermond_Popov_2016,Guermond_Popov_Tomas_2019} and by
incorporating in a roundabout way some ideas from
\cite{Gosse_Toscani_2002} and \citep{Ragusa_Guermond_Kanschat_2012}.
We propose a method that is both positivity-preserving and does not
lock in the thick diffusion limit. (The method shares some
similarities with the two-dimensional finite volume technique from
\cite[Eq.~(18)-(19)]{Buet_Despres_Emmanuel_2012}.) Being linear, and
in compliance with Godunov's theorem, the proposed algorithm is only
first-order accurate in space though.  This work is the first part of
a ongoing project aiming at developing techniques that are high-order
accurate, positivity-preserving, and robust in the diffusion limit.
The next step will be to increase the accuracy by introducing a
nonlinear process; however, since this is not the purpose of the
paper, we just mention in passing possible techniques to achieve this
goal.  This could be done in many ways; for instance, one could invoke
a smoothness indicator like in \cite[\S4.3]{Guermond_Popov_2017}, one
could use a limiting technique in the spirit of the flux transport
corrected method, or one could enforce positivity through inequality
constraints like in \cite[\S4]{Hauck_McClarren_2010}.

The paper is organized as follows.  We introduce the model problem and
the discrete setting (continuous and discontinuous finite elements) in
\S\ref{Preliminaries}.  The notion of graph viscosity, as defined in
\citep{Guermond_Popov_2016,Guermond_Popov_Tomas_2019}, is introduced
in \S\ref{Sec:graph_viscosity_positivity_locking}.  We show in this
section that the graph viscosity gives a scheme that is positive, but
the scheme locks in the diffusion regime. This section is meant to
give some perspective on the material introduced in
\S\ref{sec:the_AP_scheme}. The positive and asymptotic preserving
scheme announced above is introduced in
\S\ref{sec:the_AP_scheme}. Originality is only claimed for the
material presented in this section and the next one; the key results
are Theorem~\ref{Th:diffusion_limit_AP} and
Theorem~\ref{Th:AP_positivity}.  In \S\ref{Sec:Numerical_results} we
report numerical experiments illustrating the performance of the
proposed method. The paper finishes with \S\ref{Sec:Conclusions} where
we make concluding remarks.

\section{Preliminaries} \label{Preliminaries}
In this section, we introduce the model problem under investigation 
and some notation regarding the discretization.

\subsection{The model problem}
Let $\Dom$ be an open, bounded, connected Lipschitz domain in $\Real^3$ and let
$\calS$ be the unit sphere in $\Real^3$.  We denote by $\mes{\calS}$
the measure of $\calS$, \ie $\mes{\calS}=4\pi$.  The boundary of
$\Dom$ is denoted by $\partial\Dom$ and the outer unit normal is
denoted by $\bn$.  We want to solve the linear, one-group, radiation transport equation
\begin{subequations}
\begin{align}\bOmega\ADV\psi(\bx,\bOmega) + \sigma_t(\bx)\psi(\bx,\bOmega)
&= \sigma_s(\bx)\opsi(\bx) +q(\bx,\bOmega), && (\bx,\bOmega)\in \Dom\CROSS\calS\\
\psi(\bx,\bOmega)&=\alpha(\bx,\bOmega), && (\bx,\bOmega)\in \front_{-} \\
\opsi(\bx) &= \frac{1}{\mes{\calS}} \int_\calS \psi(\bx,\bOmega) \diff \bOmega, &&  \bx\in \Dom,
\end{align}  \label{model_pb}%
\end{subequations}%
with
$\front_{-}:=\{(\bx,\bOmega)\in \partial\Dom\times\calS \,|\,
\bOmega\SCAL\bn(\bx) <0\}$.
The independent variable $(\bx,\bOmega)$ spans $\Dom\CROSS\calS$.  The
dependent variable $\psi(\bx,\bOmega)$ is referred to as the angular
intensity or angular flux, and the quantity $\opsi(\bx)$ is called
scalar intensity or flux. The symbols $\sigma_t(\bx)$ and $\sigma_s(\bx)$ denote the total and
scattering cross sections, respectively.

We want to investigate the approximation of \eqref{model_pb} using
either continuous or discontinuous finite elements. The objective is
to construct a method that is asymptotic preserving in
the diffusion limit and positive (assuming that the boundary data, the cross
sections, and the source term are non-negative).  In order to do that, we
are going to adopt an idea from \cite{Gosse_Toscani_2002}, where a
relaxation of the so called hyperbolic heat equation is introduced,
and an idea from \cite{Ragusa_Guermond_Kanschat_2012} where, in
addition to the mesh size, the stabilization parameters of the
approximation have been made to depend on the cross sections as well.

\subsection{Angular discretization}
In order to simplify the presentation we assume that the
discretization in angle is done using a discrete ordinate
technique. The (finite) angular quadrature is denoted
$(\mu_l,\bOmega_l)_{l\in\calL}$ and is assumed to satisfy
\begin{equation}
  \sum_{l\in\calL} \mu_l = \mes{\calS} ,\quad  \sum_{l\in\calL}\mu_l\bOmega_l=\bzero,\quad \sum_{l\in\calL}
\bOmega_l |\bc \SCAL \bOmega_l|=\bzero,\quad
\sum_{l\in\calL} \mu_l \bOmega_l{\otimes} \bOmega_l = \frac{\mes{\calS}}{3} \polI, \label{angular_quadrature}
\end{equation}
for all $\bc\in\Real^3$, where $\polI$ is the $3\CROSS 3$ identity
matrix. Recall that $\mes{\calS}=4\pi$.  For further reference we also
define the set $\calA_L:=\{\bOmega_l\in\Real^3,\ l\in\calL\}$, with
$L:=\text{card}(\calL)$.

\subsection{Continuous finite elements} \label{sec:CG}
We describe in this section the Galerkin approximation of
\eqref{model_pb} with continuous finite elements. This technique is
not positive and is known to exhibit severe oscillations; it will be appropriately
stabilized in \S\ref{sec:the_AP_scheme}.

Let $\famTh$ be a shape-regular sequence of unstructured matching
meshes. For simplicity we assume that all the elements are generated
from a reference element denoted $\wK$.  The geometric transformation
mapping $\wK$ to an arbitrary element $K\in \calT_h$ is denoted
$T_K : \wK \longrightarrow K$. We now introduce a reference finite
element $(\wK,\wP,\wSigma)$, which we assume, for simplicity, to be a
Lagrange element.  We define the following scalar-valued finite
element space:
\begin{align} 
\label{eq:Xh}
P\upg(\calT_h) &=\{ v\in \calC^0(\Dom;\Real)\st 
v_{|K}{\circ}T_K \in \wP,\ \forall K\in \calT_h\}.
\end{align}  
The superscript $\upg$ is meant to remind us that the space is
conforming for the gradient operator, \eg
$P\upg(\calT_h)\subset H^1(\Dom)$. The global shape functions are
denoted by $\{\varphi_i\}_{i \in \calV}$; the associated Lagrange
nodes are denoted $\{\ba_i\}_{i \in \calV}$. We recall that the global
shape functions satisfy the partition of unity property
$\sum_{i \in \calV}\varphi_i(\bx) =1$, for all $\bx \in\Dom$. We
assume that they have positive mass
\begin{equation}
m_i:= \int_\Dom \varphi_i(\bx)\diff \bx >0, \qquad \forall i\in\calV.
\end{equation}  For
any $i\in \calV$, the adjacency list $\calI(i)$ is defined by setting
$\calI(i):=\{j\in\calV\st \varphi_i\varphi_j\not\equiv 0\}$.  The approximation space for~\eqref{model_pb}
is then defined to be
\begin{equation}
  \bP\upg(\calT_h,\calA_L) := 
\underbrace{P\upg(\calT_h)\CROSS \ldots \CROSS P\upg(\calT_h)}_{\text{$L$ times}}.
\end{equation}

Let $\sigma_{t,i}$ and $\sigma_{s,i}$ be consistent approximations of
$\sigma_t$ and $\sigma_s$ at the Lagrange node $\ba_i$. For instance
let us assume that the mesh $\calT_h$ is such that $\sigma_t$ and
$\sigma_s$ are continuous over each cell $K$ in $\calT_h$ ($\sigma_t$
and $\sigma_s$ can be discontinuous across some mesh interfaces). Let
us denote $\calT(i)=\{K\in\calT_h\st \ba_i \in K\}$. Then we can set
$\sigma_{t,i} = \frac{1}{\text{card}(\calT(i))}
\sum_{K\in\calT(i)}\sigma_{t|K}(\ba_i)$
and
$\sigma_{s,i} = \frac{1}{\text{card}(\calT(i))}
\sum_{K\in\calT(i)}\sigma_{s|K}(\ba_i)$.
For further reference we denote the absorption cross section at node
$\ba_i$ by $\sigma_{a,i}:= \sigma_{t,i}-\sigma_{s,i}$.

Let
$\bpsi_h :=
(\psi_{h,1},\ldots,\psi_{h,L})\in\bP\upg(\calT_h,\calA_L)$, with
$ \psi_{h,k} := \sum_{j\in\calV} \Psi_{ik}\varphi_j\in
P(\calT_h)$ for all $k\in\calL$, be the discrete ordinate Galerkin
approximation of~\eqref{model_pb}. The field
$\bpsi_h\in\bP\upg(\calT_h,\calA_L)$ is obtained by solving the
following set of linear equations:
\begin{subequations}
\begin{align}
\sum_{j\in\calI(i)}\!\!\Psi_{jk}\!\! \int_\Dom\!  (\bOmega_k\SCAL \GRAD \varphi_j) \varphi_i\diff \bx 
+ m_i \sigma_{t,i} \Psi_{ik} &= m_i \sigma_{s,i} \oPsi_{i} + m_i q_{ik} + b_{ik}^\partial (\alpha_{ik}^\partial-\Psi_{ik}),\\
\oPsi_{i} &= \frac{1}{\mes{\calS}} \sum_{k\in\calL} \mu_k \Psi_{ik},
\end{align}
\end{subequations}
where we have lumped the mass matrix, defined
$q_{ik}:=\frac{1}{m_i}\int_\Dom \varphi_i(\bx) q(\bx,\bOmega_k)\diff \bx$, and
set
\begin{equation}
b_{ik}^\partial = 
 m_i^\partial \frac{|\bOmega_k\SCAL \bn_i|-\bOmega_k\SCAL \bn_i}{2}.
\end{equation}  
Here $ m_i^\partial :=\int_{\front} \varphi_i(\bx) \diff s$, $\bn_i$ is
the unit normal vector (or approximation thereof) at the Lagrange node $\ba_i$, and
$\alpha_{ik}^\partial:=\alpha(\ba_i,\bOmega_k)$.  
To  refer to boundary
degrees of freedom we introduce the following set of indices:
\begin{equation}
(\calV\CROSS\calL)^\partial:=\{(j,l)\in \calV\CROSS \calL \st
  \bOmega_l\SCAL \bn_j< 0\}.
\end{equation}
For further reference we introduce
\begin{equation}
\bc_{ij}:= \int_D  \varphi_i(\bx)\nabla\,\varphi_j(\bx)\,\diff \bx.
\end{equation}
With this notation, the discrete system is rewritten as follows for all $(i,k)\in\calV\CROSS\calL$
\begin{equation}
 \sum_{j\in\calI(i){\setminus}\{i\}}\!\! \bOmega_k\SCAL\bc_{ij} (\Psi_{jk} - \Psi_{ik})
+ m_i \sigma_{t,i} \Psi_{ik} = m_i \sigma_{s,i} \oPsi_{i} + m_i q_{ik}
+b_{ik}^\partial (\alpha_{ik}^\partial-\Psi_{ik}). \label{cG_system}
\end{equation}
Notice that here we have used the partition of unity property which
implies that $\sum_{j\in\calI(i)} \bc_{ij} =\bzero$.

\begin{remark}[Boundary conditions] We have imposed the boundary
  condition weakly in \eqref{cG_system} by using the penalty technique
  usually invoked in the context of discontinuous Galerkin
  approximations, but one can also enforce the boundary conditions
  strongly. In that case one sets $b_{ik}^\partial=0$ and one adds  the
  equations $\Psi_{ik}=\alpha_{ik}^\partial$ to \eqref{cG_system} for all
  $(i,k)\in(\calV\CROSS\calL)^\partial$.
\end{remark}

As mentioned above, the linear system~\eqref{cG_system} has no
positivity property. We are going to remedy this problem in
\S\ref{sec:the_AP_scheme}.

\subsection{Discontinuous finite elements}\label{sec:DG}
We briefly describe in this section the discontinuous Galerkin 
approximation of~\eqref{model_pb} with the centered numerical flux.

We use the same notation as in \S\ref{sec:CG} for the shape-regular
sequence of unstructured matching meshes $\famTh$. We also introduce a
reference finite element $(\wK,\wP,\wSigma)$. This may not be a
Lagrange element. We define the following scalar-valued broken finite element
space:
\begin{align} 
\label{eq:dG_Xh}
P\upb(\calT_h) &=\{ v\in L^1(\Dom;\Real)\st 
v_{|K}{\circ}T_K \in \wP,\ \forall K\in \calT_h\}.
\end{align} 
The superscript $\upb$ is meant to remind us that the space is broken,
\ie the members of $P\upb(\calT_h)$ can be discontinuous across the
mesh interfaces.  We denote by $\{\varphi_i\}_{i\in\calV}$ the
collection of the global shape functions generated from the reference
shape functions. The support of each shape function is restricted to
one mesh cell only.  We assume that all the shape functions have a
positive mass
\begin{equation}
m_i:=\int_\Dom \varphi_i \dif x >0,\quad \forall i\in\calV.
\end{equation}

We
introduce the following adjacency sets:
\begin{align}\label{def:setpartialK}
\calI(K) := \big\{ i \in \calV \st \varphi_{i|K} \not \equiv 0 \big\},\qquad 
  \calI(\partial K) := \big\{ i \in \calV \st \varphi_{i|\partial K} \not \equiv 0 \big\}.
\end{align}
Note that $\calI(\partial K)$ not only includes indices of shape
functions with support in $\calI(K)$, but this set also includes
indices of shape functions that do not have support in $K$. More
precisely $\calI(\partial K)$ is the union of two disjoint sets
$\calI(\partial K\upi)$ and $\calI(\partial K\upe)$ defined as
\begin{align}
\calI(\partial K\upi) 
&:= \big\{ i \in \calI(K) \ \big| \ \varphi_{i|\partial K} \not \equiv 0 \big\}, 
\qquad  \calI(\partial K\upe) 
:= \calI(\partial K)\backslash \calI(\partial K\upi).
\end{align}
For any $i\in \calV$, let $K\in\calT_h$ be such that $i\in\calI(K)$;
then we define the adjacency set $\calI(i)$ to be the collection of
the indices $j\in\calV$ such that either $j\in \calI(K)$ and $\varphi_i\varphi_j|_{K}\not\equiv 0$, or
$j\in \calI(\partial K\upe)$ and
$\varphi_i\varphi_j|_{\partial K}\not\equiv 0$.

Let $K\in\calT_h$. We finally assume that the reference finite element
is such that the sets of shape functions
$\{\varphi_j\}_{j \in \calI(K)}$ form a partition of unity over $K$,
and the shape functions $\{\varphi_j\}_{j \in \calI(\partial K\upi)}$,
$\{\varphi_j\}_{j \in \calI(\partial K\upe)}$ form partitions of unity
over $\partial K$, \ie
\begin{align}
\label{boundary_partition_unity}
 \sum_{j \in \calI(K)} \varphi_{j|K} = 1, \qquad 
  \sum_{j \in \calI(\partial K\upi)} \varphi_{j|\partial K} = 1,
\ \text{ and  } \sum_{j \in \calI(\partial K\upe)} \varphi_{j|\partial K} = 1.
\end{align}
Let $i\in V$, $j\in\calI(i)$, let us set
\begin{equation}
\label{def_of_cij_bulk_and_bnd}
\bc_{ij}^{K}:= \int_{K}\varphi_i \GRAD\varphi_j \diff \bx,\qquad
\bc_{ij}^{\partial K} := \tfrac{1}{2} \int_{\partial K} \varphi_{j} \varphi_i \bn_K \diff s,
\end{equation}
and let us define the vector $\bc_{ij}$ as follows:
\begin{equation}
\label{def:cijdG}
\bc_{ij} :=
 \begin{cases} 
\bc_{ij}^K                               &\text{if }j \in \calI(K)\backslash\calI(\partial K^{\mathsf{i}}), \\
\bc_{ij}^K-\bc_{ij}^{\partial K} &\text{if }j \in \calI(\partial K^{\mathsf{i}}), \\
\bc_{ij}^{\partial K}                  &\text{if }j \in \calI(\partial K^{\mathsf{e}}).
\end{cases}
\end{equation}
The partition of unity property~\eqref{boundary_partition_unity}
implies that $\sum_{j\in\calI(i)} \bc_{ij} =\bzero$ (see for instance
\citep[Lem.~4.1]{Guermond_Popov_Tomas_2019}).

Let us introduce the discrete broken space
\begin{equation}
  \bP\upb(\calT_h,\calA_L) := \underbrace{P\upb(\calT_h)\CROSS \ldots \CROSS P\upb(\calT_h)}_{\text{$L$ times}}.
\end{equation}
Let us denote by
$\bpsi_h :=
(\psi_{h,1},\ldots,\psi_{h,L})\in\bP\upb(\calT_h,\calA_L)$, with
$ \psi_{h,k} := \sum_{j\in\calV\CROSS\calL} \Psi_{jk}\varphi_j\in
P\upb(\calT_h)$, the dG approximation of~\eqref{model_pb} using the
centered flux.  The field $\bpsi_h\in\bP\upb(\calT_h,\calA_L)$ is
defined to be the solution of
 \begin{equation}
 \sum_{j\in\calI(i){\setminus}\{i\}}\!\! \bOmega_k\SCAL\bc_{ij} (\Psi_{jk} - \Psi_{ik})
+ m_i \sigma_{t,i} \Psi_{ik} = m_i \sigma_{s,i} \oPsi_{i} + m_i q_{ik}
+b_{ik}^\partial (\alpha_{ik}^\partial-\Psi_{ik}). \label{dG_system}
\end{equation}
We insist here that we are using the centered flux; there is no
upwinding. The proper stabilization will be introduced in
\S\ref{sec:the_AP_scheme}.

\begin{remark}[Definition of $\sigma_{t,i}$ and $\sigma_{s,i}$]
  The definition of the coefficients $\sigma_{t,i}$ and $\sigma_{s,i}$
  depend on the definition of the shape functions.  If the shape
  functions are nodal-based (\ie Lagrange polynomials) then one can
  take $\sigma_{t,i} =\sigma_{t|K}(\ba_i)$, where $K$ contains the
  support of $\varphi_i$ and $\ba_i$ is the Lagrange node associated
  with $\varphi_i$, and we recall that we denote
$\sigma_{a,i}:= \sigma_{t,i}-\sigma_{s,i}$.
\end{remark}

\section{Graph viscosity, positivity, and
  locking} \label{Sec:graph_viscosity_positivity_locking}

In order to give some perspective, we start by introducing a mechanism
that ensures positivity but fails to be robust in the diffusion limit.
A correction that makes the method asymptotic-preserving in the
diffusion limit is introduced in \S\ref{sec:the_AP_scheme}.

\subsection{Positivity} \label{Sec:positivity}

Our starting point is the algebraic system \eqref{cG_system} or
\eqref{dG_system}, which we call Galerkin, or centered, or inviscid
approximation. We are not going to make any distinction between the
continuous and the discontinuous Galerkin approximations.  The
discrete space are henceforth denoted $P(\calT_h)$ and $\bP(\calT_h)$,
\ie we have removed the superscripts $\upg$ and $\upb$.  We consider
the following linear system: Find
$\bpsi_h = \sum_{i\in\calV}(\Psi_{i1},\ldots,\Psi_{iL})\varphi_i\in
\bP(\calT_h)$ so that the following holds for all $(i,k)\in\calV\CROSS\calL$:
\begin{equation}
 \sum_{j\in\calI(i){\setminus}\{i\}}\!\! \bOmega_k\SCAL\bc_{ij} (\Psi_{jk} - \Psi_{ik})
+ m_i \sigma_{t,i} \Psi_{ik} = m_i \sigma_{s,i} \oPsi_{i} + m_i q_{ik}
+b_{ik}^\partial (\alpha_{ik}^\partial-\Psi_{ik}), \label{galerkin_system}
\end{equation}
where we recall that $\sum_{j\in\calI(i)} \bc_{ij}=0$ for all
$i\in\calV$.  Taking inspiration from~\cite{Guermond_Popov_2017}, we
introduce the coefficient $d_{ij}^k$ defined by setting 
\begin{equation}
  d_{ij}^k = \max(\max(\bOmega_k\SCAL \bc_{ij},0),\max(\bOmega_k\SCAL \bc_{ji},0)). \label{def_dij}
\end{equation}
Then 
we perturb \eqref{galerkin_system} as follows:
\begin{equation}
 \sum_{j\in\calI(i){\setminus}\{i\}}\!\!\!\! (\bOmega_k\SCAL\bc_{ij} - d_{ij}^k) (\Psi_{jk} - \Psi_{ik})
+ m_i \sigma_{t,i} \Psi_{ik} = m_i \sigma_{s,i} \oPsi_{i} + m_i q_{ik}
+b_{ik}^\partial (\alpha_{ik}^\partial-\Psi_{ik}).
\label{dij_system}
\end{equation}
The extra term
$\sum_{j\in\calI(i){\setminus}\{i\}} -d_{ij}^k (\Psi_{jk} -
\Psi_{ik})$
is a graph viscosity since it acts on the connectivity graph of the
degrees of freedom. Notice that this perturbation is first-order
consistent since it vanishes if $\Psi_{jk} = \Psi_{ik}$ for all
$j\in\calI(i)$.  In one dimension on a nonuniform mesh, where the
adjacency list is $\{i-1,i,i+1\}$, we have
$d_{ij}^k=\frac{|\bOmega_k|}{2}$ both for continuous piecewise linear
finite elements and for piecewise constant discontinuous elements; as
a result, we have
$\sum_{j\in\calI(i){\setminus}\{i\}} -d_{ij}^k (\Psi_{jk} - \Psi_{ik})
= -\frac{|\bOmega_k|}{2} (\Psi_{i-1,k} - 2\Psi_{ik} + \Psi_{i+1,k})$,
which is the expression one expects from an artificial viscosity
term. Further insight on the graph viscosity is given in
Remark~\ref{Rem:dg0} in the context of the dG0 setting. The following
result is the key motivation for introducing the graph viscosity.

\begin{lemma}[Minimum/Maximum principle] \label{Lem:min_principle} Let
  $d_{ij}^k$ be defined in~\eqref{def_dij}. Let
  $(\Psi_{ik})_{(i,k)\in\calV\CROSS \calL}$ be the solution to
  \eqref{dij_system}. Let
  $\Psi^{\min}:=\min_{(i,k)\in \calV\CROSS\calL} \Psi_{ik}$ and
  $\Psi^{\max}:=\max_{(i,k)\in \calV\CROSS\calL} \Psi_{ik}$. Let
  $(i_0,k_0), (i_1,k_1)\in \calV\CROSS \calL$ be so that $\Psi_{i_0k_0} =\Psi^{\min}$ and
  $\Psi_{i_1k_1} = \Psi^{\max}$.
  \begin{enumerate}[(i)]
  \item \label{item1:Lem:min_principle} Assume that
    $\min_{(j,l)\in\calV\CROSS\calL} (\sigma_{a,j}+b_{jl}^\partial)>0$. Then 
    \begin{equation}
      \frac{m_{i_0} q_{i_0k_0} + b_{i_0k_0}^\partial \alpha_{i_0k_0}^\partial}{m_{i_0}
        \sigma_{a,i_0} + b_{i_0k_0}^\partial} \le \Psi^{\min}\le 
     \Psi^{\max}\le  \frac{m_{i_1} q_{i_1k_1} + b_{i_1k_1}^\partial \alpha_{i_1k_1}^\partial}{m_{i_1}
        \sigma_{a,i_1} + b_{i_1k_1}^\partial}.
      \quad 
   \end{equation}
 \item \label{item2:Lem:min_principle} Otherwise, assume that for all
   $i\in\calV$ such that $\sigma_{a,i}=0$ and $b_{ik}^\partial=0$ the
   definition of $d_{ij}^k$ is slightly modified so that
   $\bOmega_k\SCAL\bc_{ij} < d_{ij}^k$ for all $j\in\calI(i)$ (instead of $\bOmega_k\SCAL\bc_{ij} \le d_{ij}^k$). If
   $0\le \min_{(i,k)\in \calV\CROSS\calL} q_{ik}$ and
   $0\le \min_{(i,k)\in (\calV\CROSS\calL)^\partial} \alpha_{ik}^\partial$, 
   then $0\le \Psi^{\min}$.
 \item \label{item3:Lem:min_principle} Moreover, under the same
   assumptions on $d_{ij}^k$ as in \eqref{item2:Lem:min_principle}, if
   $\max_{(i,k)\in \calV\CROSS\calL} q_{ik}\le 0$, then
   $\Psi^{\max} \le \max_{(i,k)\in (\calV\CROSS\calL)^\partial}
   \alpha_{ik}^\partial$.
  \end{enumerate}
\end{lemma}
\begin{proof}
  Proof of \eqref{item1:Lem:min_principle}.  We start by assuming that
  $\min_{j\in\calV} (\sigma_{t,j}- \sigma_{s,j})>0$. Let
  $(i_0,k_0)\in \calV\CROSS \calL$ be the indices of the degree of freedom
  where the minimum is attained; that is, $\Psi_{ik}\ge \Psi_{i_0k_0}$
  for all $(i,k)\in\calV\CROSS \calL$. Then using that
  \[
  \bOmega_k\SCAL\bc_{ij} - d_{ij}^k \le \max(\bOmega_k\SCAL\bc_{ij},0) -
  d_{ij}^k  \le 0,
 \]
together with $\Psi_{jk_0} - \Psi_{i_0k_0} \ge 0$ for all $j\in\calI(i_0)$, and
 $\Psi_{i_0k_0} \le \oPsi_{i_0} $, we infer that
\begin{align*}
   m_{i_0} \sigma_{s,i_0} &\Psi_{i_0k_0} + m_{i_0} q_{i_0k_0} + b_{i_0k_0}^\partial (\alpha_{i_0k_0}^\partial-\Psi_{i_0k_0}) \\
   &\le m_{i_0} \sigma_{s,i_0} \oPsi_{i_0} + m_{i_0} q_{i_0k_0}  
+ b_{i_0k_0}^\partial (\alpha_{i_0k_0}^\partial-\Psi_{i_0k_0}) \\
  & = \!\!\sum_{j\in\calI(i_0){\setminus}\{i_0\}}\!\!
   (\bOmega_{k_0}\SCAL\bc_{i_0j} - d_{i_0j}^{k_0}) (\Psi_{jk_0} - \Psi_{i_0k_0})
   + m_{i_0} \sigma_{t,i_0} \Psi_{i_0k_0} \le m_{i_0} \sigma_{t,i_0} \Psi_{i_0k_0}.
 \end{align*}
 Hence
 $m_{i_0} q_{i_0k_0} + b_{i_0k_0}^\partial \alpha_{i_0k_0}^\partial \le (m_{i_0}
 \sigma_{a,i_0} + b_{i_0k_0}^\partial)\Psi_{i_0k_0}$.
 The assertion follows readily since $b_{i_0k_0}^\partial\ge 0$ implies that $m_{i_0}
 \sigma_{a,i_0} + b_{i_0k_0}^\partial>0$. The proof of the other assertion, regarding $\Psi^{\max}$, is analogous.

 Proof of \eqref{item2:Lem:min_principle} assuming that
 $0\le \min_{(i,k)\in \calV\CROSS\calL} q_{ik}$ and
 $0\le \min_{(i,k)\in (\calV\CROSS\calL)^\partial}
 \alpha_{ik}^\partial$.  From part (i) we have
 $m_{i_0} q_{i_0} +b_{i_0k_0}^\partial \alpha_{i_0k_0}^\partial\le (m_i\sigma_{a,i_0} +
 b_{i_0k_0}^\partial)\Psi_{i_0k_0}$. So, we need to prove that
 $\Psi_{i_0k_0}\ge 0$ only in the case $\sigma_{a,i_0} =0 $ and
 $b_{i_0k_0}^\partial=0$. Assuming that $\sigma_{a,i_0} =0 $ and
 $b_{i_0k_0}^\partial=0$, we have from part
 \eqref{item1:Lem:min_principle} the following inequality
 \[
   0 \le m_{i_0} q_{i_0k_0} 
      \le \!\!\sum_{j\in\calI(i_0){\setminus}\{i_0\}}\!\!
   (\bOmega_{k_0}\SCAL\bc_{i_0j} - d_{i_0j}^{k_0}) (\Psi_{jk_0} - \Psi_{i_0k_0}) \le 0.
 \]
 The assumption $\bOmega_k\SCAL\bc_{i_0j} - d_{i_0j}^{k_0}<0$ for all
 $j\in\calI(i_0)$, implies that $\Psi_{jk_0} - \Psi_{i_0k_0}=0$ for
 all $j\in\calI(i_0)$. Therefore, we conclude that the global minimum
 is attained not only at the degree of freedom $(i_0,k_0)$ but also in
 the whole neighborhood, \ie for all $j\in\calI(i_0)$. Repeating the
 above argument for a global minimum at $(j, k_0)$ for all
 $j\in\calI(i_0)$, we derive that the global minimum is either
 nonnegative (if $m_{j} \sigma_{a,j} + b_{jk_0}^\partial>0$) or again
 attained in the whole neighborhood of $j$, \ie for all
 $s\in\calI(j)$. This process can terminate in two ways: (i) either
 the global minimum is nonnegative at some $j$, \ie
 $m_{j} \sigma_{a,j} + b_{jk_0}^\partial>0$; 
 (ii) or the global
 minimum is attained at all of the degrees of freedom topologically
 connected to $i_0$.  In the first case we have proved the
 non-negativity, in the second case we have that
 $\Psi_{jk_0}= \Psi_{i_0k_0}$ for all $j$ in the same connected
 component as $i_0$, which is the entire set $\calV$ since $\calT_h$
 is connected (because $\Dom$ is connected).  However, for any fixed
 $k_0$ there exists $j$ such that $\Psi_{jk_0}$ is on the inflow
 boundary for $\bOmega_{k_0}$, that is $b_{jk_0}^\partial>0$, and we
 conclude $\Psi_{i_0k_0} =\Psi_{jk_0}\ge 0$.

 Proof of \eqref{item3:Lem:min_principle} assuming that
 $\min_{(i,k)\in \calV\CROSS\calL} q_{ik}\le 0$. By
 proceeding as in Step~\eqref{item1:Lem:min_principle}, we infer that
 \begin{align*}
   m_{i_1} \sigma_{s,i_1} &\Psi_{i_1k_1} + m_{i_1} q_{i_1k_1} + b_{i_1k_1}^\partial (\alpha_{i_1k_1}^\partial-\Psi_{i_1k_1}) \\
  & \ge  \!\!\sum_{j\in\calI(i_1){\setminus}\{i_1\}}\!\!
   (\bOmega_{k_1}\SCAL\bc_{i_1j} - d_{i_1j}^{k_1}) (\Psi_{jk_1} - \Psi_{i_1k_1})
   + m_{i_1} \sigma_{t,i_1} \Psi_{i_1k_1} ,
 \end{align*}
 \ie
 $(m_{i_1} \sigma_{a,i_1} + b_{i_1k_1}^\partial )\Psi_{i_1k_1} \le
 m_{i_1} q_{i_1k_1} + b_{i_1k_1}^\partial \alpha_{i_1k_1}^\partial$;
 which implies $\Psi_{i_1k_1} \le \alpha_{i_1k_1}^\partial$ if $b_{i_1k_1}^\partial>0$. 
 Hence we just
 need to consider the case $b_{i_1k_1}^\partial=0$. In that case
 $0 \ge \!\!\sum_{j\in\calI(i_1){\setminus}\{i_1\}}
 (\bOmega_{k_1}\SCAL\bc_{i_1j} - d_{i_1j}^{k_1}) (\Psi_{jk_1} -
 \Psi_{i_1k_1}) \ge 0$ and $\Psi_{jk_1} =\Psi_{i_1k_1}$ for all
 $j\in \calI(i_1)$. Then we proceed as in
 Step~\eqref{item2:Lem:min_principle} until we reach a dof $j$ that is on
 the inflow boundary for $\bOmega_{k_1}$, \ie $b_{jk_1}^\partial>0$.
 The $\Psi^{\max} = \Psi_{i_1k_1} = \Psi_{jk_1} \le \alpha_{j,k_1}^\partial$.
\end{proof}

\begin{remark}[dG0] \label{Rem:dg0} To give some insight
  about~\eqref{def_dij} to the reader who is familiar with the dG
  formulation of the radiation transport equation, we now interpret
  the graph viscosity in terms of numerical flux. Assume that
  $P\upb(\calT_h)$ is composed of piecewise constant polynomials. In
  this case the indices $i\in\calV$ coincide with the enumeration of
  the cells in $\calT_h$. Let $K_i\in\calT_h$ be a cell and let
  $(K_j)_{j\in\calI(i)}$ be all the cells that share a face with
  $K_i$, then recalling~\eqref{def:cijdG}, we have
  $\bc_{ii}=\int_{K_i} \varphi_i \GRAD \varphi_i \diff \bx - \frac12
  \int_{\partial K_i} \varphi_i^2 \bn_K \diff s$
  and
  $\bc_{ij} = \frac12 \int_{\partial K_i} \varphi_i\varphi_j \bn_K
  \diff s$
  for all $j\in\calI(i){\setminus}\{i\}$.  Let us set
  $\psi_{h,k}(\bx) = \sum_{j\in\calV} \Psi_{jk} \varphi_j \in
  P\upb(\calT_h)$.
  Let us denote $\psi_{h,k}\upe$ and $\psi_{h,k}\upi$, respectively,
  the exterior trace and the interior trace of $\psi_{h,k}$ on
  $\partial K_i$. Then
  \begin{align*}
    \sum_{j\in\calI(i){\setminus}\{i\}}
& ( \bOmega_k\SCAL\bc_{ij} -d_{ij}^k) ( \Psi_{jk} - \Psi_{ik}) =
\int_{K_i}   \varphi_i\bOmega_k\SCAL\GRAD \psi_{h,k}(\bx) \diff \bx \\
 &+ \int_{\partial K_i} \frac12 (\psi_{h,k}\upe-\psi_{h,k}\upi) \varphi_i \bOmega_k\SCAL\bn_K \diff s
   - \int_{\partial K_i} \frac12 (\psi_{h,k}\upe-\psi_{h,k}\upi) \varphi_i |\bOmega_k\SCAL\bn_K| \diff s   \\
 & = -\int_{K_i}   \psi_{h,k}(\bx)\bOmega_k\SCAL\GRAD\varphi_i  \diff \bx \\
 &+ \int_{\partial K_i} \Big(\frac12 (\psi_{h,k}\upe+\psi_{h,k}\upi)  \bOmega_k\SCAL\bn_K 
 + \frac12 (\psi_{h,k}\upi-\psi_{h,k}\upe) \varphi_i |\bOmega_k\SCAL\bn_K|\Big)\varphi_i  \diff s   
  \end{align*}
  Hence the dG numerical flux is
  $\frac12 (\psi_{h,k}\upe+\psi_{h,k}\upi) \bOmega_k\SCAL\bn_K +
  \frac12 (\psi_{h,k}\upi-\psi_{h,k}\upe) \varphi_i
  |\bOmega_k\SCAL\bn_K|$, and we recognize the standard upwind flux.
  In conclusion, in the dG0 context, the system~\eqref{dij_system}
  with $d_{ij}^k$ defined in~\eqref{def_dij} simply corresponds to the
  standard upwinding approximation.
\end{remark}

\subsection{Locking} \label{Sec:locking} 
Unfortunately, as reported numerous times in the literature, just
enforcing positivity in a scheme does not prevent locking.  Actually
the approximation~\eqref{dij_system} with~\eqref{def_dij} locks in the
diffusive regime.  More precisely, let $\epsilon>0$ and let us
consider the following rescaled version of the
problem~\eqref{model_pb}:
\begin{subequations}\label{model_pb_epsilon} \begin{align}
\bOmega\ADV\psi^\epsilon(\bx,\bOmega) 
+ \frac{\sigma_t(\bx)}{\epsilon}\psi^\epsilon(\bx,\bOmega)
  &= \frac{\sigma_s(\bx)}{\epsilon}\opsi^\epsilon(\bx) 
+\epsilon q(\bx,\bOmega),&&  (\bx,\bOmega)\in \Dom\CROSS\calS,
                      \\
\psi^\epsilon(\bx,\bOmega)&= \alpha(\bx,\bOmega),  && (\bx,\bOmega)\in \front_-,
\end{align}
\end{subequations}  
with the additional assumption that
$\frac{\sigma_t(\bx)-\sigma_s(\bx)}{\epsilon} = \epsilon
\sigma_a(\bx)$.
Let $\bpsi_h^\epsilon$ be the discrete ordinate approximation to the
solution of~\eqref{model_pb_epsilon} with $d_{ij}^k$ defined in~\eqref{def_dij}:
\begin{multline}
  \sum_{j\in\calI(i){\setminus}\{i\}} (\bOmega_k\SCAL\bc_{ij} -d_{ij}^k) (\Psi_{jk}^\epsilon - \Psi_{ik}^\epsilon)
  + \epsilon m_i \sigma_{a,i}\Psi_{ik}^\epsilon \\
  = m_i \frac{\sigma_{s,i}}{\epsilon}(\oPsi_{i}^\epsilon - \Psi_{ik}^\epsilon) + m_i \epsilon q_{ik}
+ b_{ik}^\partial (\alpha_{ik}^\partial -\Psi_{ik}).
 \label{dij_system_epsilon}
\end{multline}

\begin{proposition}[Locking] 
  Let the graph viscosity $d_{ij}^k$ be defined in~\eqref{def_dij}.
  Assume that $\min_{i,j} \sum_{k\in\calL}\mu_kd_{ij}^k>0$. If the
  boundary conditions are homogeneous, \ie $\alpha_{ik}^\partial=0$,
  then
  $\lim_{\epsilon\to 0} (\Psi_{jk}^\epsilon - \oPsi_{i}^\epsilon)=0$
  for all $i,j\in\calV$ and all $k\in\calL$.
 \end{proposition}
\begin{proof}
  To avoid losing the reader who is not familiar with functional
  analysis techniques, we are going to proceed formally. A rigorous
  proof can be done by proceeding as in
  \cite[\S4]{Guermond_Kanschat_2010}. Using Landau's notation, let us
  introduce the formal asymptotic expansion
  $\bpsi_h^\epsilon=\bpsi_h^\epsilon+\epsilon
  \bpsi_h^\epsilon+\epsilon^2 \bpsi_h^\epsilon + \calO(\epsilon^3).$
  Inserting this expansion into \eqref{dij_system_epsilon} gives
  \begin{align*}
    0&= m_i\sigma_{s,i}(\oPsi_{i}^0 - \Psi_{ik}^0), \\
    \sum_{j\in\calI(i){\setminus}\{i\}} (\bOmega_k\SCAL\bc_{ij} -d_{ij}^k) (\Psi_{jk}^0 - \Psi_{ik}^0)
     &= m_i \sigma_{s,i}(\oPsi_{i}^1 - \Psi_{ik}^1) + b_{ik}^\partial (\alpha_{ik}^\partial-\Psi_{ik}^0).
   \end{align*}
   The first equation shows that $\oPsi_{i}^0 = \Psi_{ik}^0$. Then
   integrating the second equation with respect to the angles gives
\[
     \sum_{j\in\calI(i){\setminus}\{i\}} (\oPsi_{j}^0 - \oPsi_{i}^0)
     \bigg(\Big(\sum_{k\in\calL}\mu_k\bOmega_k)\Big)\SCAL\bc_{ij} -\sum_{k\in\calL}\mu_kd_{ij}^k\bigg)=
m_i^\partial \frac14 \frm_i^\partial  - m_i^\partial \delta_i^\partial  \oPsi_{i}^0.
   \]
   where
   $\frm_i^\partial := \frac{4}{\mes{\calS}}\sum_{k\in\calL_i^-}
   \mu_k \alpha_{ik}^\partial |\bOmega_k\SCAL\bn_i|$
   and
   $\delta_i^\partial := \frac{1}{\mes{\calS}}\sum_{k\in\calL_i^-}
   \mu_k |\bOmega_k\SCAL\bn_i|$ with $\calL_i^-:=\{k\in\calL\st \bOmega_k\SCAL\bn_i<0\}$.
   (Note that the continuous counterparts of the coefficients $\frm_i^\partial$ and
   $\delta_i^\partial $ are
   $\frac{4}{\mes{\calS}}\int_{\bOmega\SCAL\bn_i<0}
   \alpha(\ba_i,\bOmega) |\bOmega\SCAL\bn_i|\diff \bOmega$
   and
   $\frac{1}{\mes{\calS}}\int_{\bOmega\SCAL\bn_i<0}|\bOmega\SCAL\bn_i|\diff
   \bOmega=\frac14$,
   respectively.)  Setting
   $\gamma_{ij}:=\sum_{k\in\calL}\mu_kd_{ij}^k$, the assumptions on
   the angular quadrature~\eqref{angular_quadrature} imply that
   $\sum_{k\in\calL}\mu_k\bOmega_k=\bzero$; hence,
   $ \sum_{j\in\calI(i){\setminus}\{i\}} \gamma_{ij} (\oPsi_{j}^0 -
   \oPsi_{i}^0) = m_i^\partial(\frac14 \frm_i^\partial  - \delta_i^\partial  \oPsi_{i}^0)$.
Let us assume now that $\alpha_{ik}^\partial=0$, and let us
   multiplying this equation by $2 \oPsi_{i}^0$, then
\[
  \sum_{j\in\calI(i){\setminus}\{i\}} \gamma_{ij} (\oPsi_{j}^0)^2 
+ \gamma_{ij}(\oPsi_{j}^0-\oPsi_{i}^0)^2 - \gamma_{ij} ( \oPsi_{i}^0)^2= -2 m_i^\partial\delta_i^\partial  (\Psi_{i}^0)^2.
\]
Now we observe that $\gamma_{ij}=\gamma_{ji}$, and we sum the above
identity over $i\in\calV$. This yields
$\sum_{i\in\calV} \Big(2 m_i^\partial  \delta_i^\partial  (\Psi_{i}^0)^2 + \sum_{j\in\calI(i){\setminus}\{i\}}
\gamma_{ij}(\oPsi_{j}^0-\oPsi_{i}^0)^2\Big) = 0$. This in turn implies that
$\oPsi_{j}^0=\oPsi_{i}^0$ for all $i,j\in\calV$ since
$\min_{i,j} \gamma_{ij}>0$; that is, there is locking.
 \end{proof}

\section{An asymptotic preserving scheme} \label{sec:the_AP_scheme}
The goal of this section is to introduce the asymptotic preserving
method mentioned in the introduction of the paper. 
This scheme is
somewhat discretization agnostic since it is solely based on the
algebraic formulations~\eqref{cG_system} and \eqref{dG_system}.

\subsection{Preliminary notation}
In the rest of the paper we use the following notation:
\begin{subequations}%
\begin{align}%
\calL_i^-&:=\{k\in\calL\st \bOmega_k\SCAL\bn_i<0\}, &&
\delta_i^\partial := \frac{1}{\mes{\calS}}\sum_{k\in\calL_i^-}
   \mu_k |\bOmega_k\SCAL\bn_i| \\
\frm_i^\partial &:= \frac{4}{\mes{\calS}}\sum_{k\in\calL_i^-}
   \mu_k \alpha_{ik}^\partial |\bOmega_k\SCAL\bn_i|, &&
\bfrM_i^\partial := \frac{1}{\mes{\calS}}\sum_{k\in\calL_i^-}
   \mu_k \alpha_{ik}^\partial |\bOmega_k\SCAL\bn_i|\bOmega_k.
\end{align}%
\end{subequations}%
We also denote
$\calV\upint=\{i\in\calV\st \varphi_{i|\front}\equiv 0\}$ and
$\calV\upbnd=\calV{\setminus} \calV\upint$. For further reference we define
$\calL(\bx)^-:=\{k\in\calL\st \bOmega_k\SCAL\bn(\bx)<0\}$,
$\frm^\partial(\bx) := \frac{4}{\mes{\calS}}\sum_{k\in\calL(\bx)^-}
   \mu_k \alpha(\bx,\bOmega_k)^\partial |\bOmega_k\SCAL\bn(\bx)|$, and
$\bfrM^\partial(\bx) := \frac{1}{\mes{\calS}}\sum_{k\in\calL(\bx)^-}
   \mu_k \alpha(\bx,\bOmega_k)^\partial |\bOmega_k\SCAL\bn(\bx)|\bOmega_k$
for all $\bx\in\front$.

Now, depending whether one uses (or prefer using) continuous finite
elements or discontinuous finite elements, we introduce two sets of
coefficients.  In the context of continuous finite elements we set
\begin{equation}
 c_{ij}^{\textup{g,d}} =\int_{\Dom}\frac{1}{\tsigma_s(\bx)}
\GRAD\varphi_i(\bx)\SCAL\GRAD\varphi_j(\bx)\diff\bx, \qquad i,j\in\calV,
\end{equation}
where
$\tsigma_s(\bx)= \max(\sigma_s(\bx),\varepsilon
\max(\frac{1}{\diam(D)}, \|\sigma_s\|_{L^\infty(\Dom)}))$
with $\varepsilon=10^{-14}$. The quantity $\tsigma_s$ is introduced to
avoid divisions by zero.  For discontinuous finite elements of degree
$1$ or larger we proceed as follows.  We assume for simplicity that
$\tsigma_s$ is constant over each mesh cell and denote
$\tsigma_{s,K}:= \tsigma_{s|K}$ for all cells $K$. Let $K\in\calT_h$
and let $\calF_K\upint$ be the set in the faces of $K$ that are not on
$\front$; that is, $F\in\calF_K\upint$ if there exists $K'\in\calT_h$, $K'\ne K$, such that 
$F:=K\cap K'$. 
For every $F\in \calF_K\upint$,
we define $\tsigma_{s,F} =
\frac{2\tsigma_K\tsigma_{K'}}{\tsigma_K+\tsigma_{K'}}$
and $h_F:=\text{diam}(\mes{F})$.  Let $v\in\bP\upb(\calT_h)$ and let
$v_K$, $v_{K'}$ be the restrictions of $\bv$ on $K$ and $K'$
respectively; we define the weighted average of $v$ across $F\in\calF_K\upint$ as
follows:
$\avg{v}:=\frac{\tsigma_K}{\tsigma_K+\tsigma_{K'}}v_{K|F} +
\frac{\tsigma_{K'}}{\tsigma_K+\tsigma_{K'}}v_{K'|F}$.
The jump of $v$ across $F\in\calF_K\upint$ is defined by setting
$\jump{v}:=v_K-v_K'$. We now define for all $j\in\calI(i)$
\begin{multline}
 c_{ij}^{\textup{b,d}} =\int_{K}\frac{1}{\tsigma_s}
 \GRAD\varphi_i\SCAL\GRAD(\varphi_{j|K}) \diff\bx
 +\gamma \sum_{F\in\calF_K\upint} \frac{1}{\tsigma_{s,F} h_F} \int_F \jump{\varphi_i} \jump{\varphi_j} \diff s \\
 - \sum_{F\in\calF_K\upint} \int_F \left(\avg{\tfrac{1}{\tsigma_s}\GRAD\varphi_i}\SCAL \bn_{K}\jump{\varphi_j}
   + \avg{\tfrac{1}{\tsigma_s}\GRAD\varphi_j}\SCAL \bn_{K}\jump{\varphi_i}\right) \diff s,
\end{multline}
where $\gamma$ is a user-defined constant of order $1$, and with the convention that 
$\varphi_{j|K}=0$ if $j\in \calI(\partial K\upe)$. Denoting by
$c_{ij}\upd$ either $c_{ij}^{\textup{g,d}}$ or
$c_{ij}^{\textup{b,d}}$, depending on the context, and with
$v_h:=\sum_{j\in\calV}\sfV_j\varphi_j$ and
$w_h:=\sum_{j\in\calV}\sfW_j\varphi_j$, the bilinear form
$a:P(\calT_h)\CROSS P(\calT_h)\to \Real$ defined by
\begin{equation}
a(v_h,w_h) := \frac13 \sum_{i,j\in\calV} c_{ij}\upd \sfV_j \sfW_i
\end{equation}
 is the
discrete weak form of the operator $-\DIV(\frac{1}{3\sigma_s}\GRAD v)$
which naturally appears in the diffusion limit of \eqref{model_pb}.
Notice that the partition of unity property implies that
$\sum_{j\in \calI(i)} c_{ij}\upd =0$; hence, we can also write 
$a(v_h,w_h) = \frac13\sum_{i\in\calV}\sum_{j\in\calV{\setminus}\{i\}} c_{ij}\upd (\sfV_j-\sfV_i) \sfW_i$.

\subsection{Description of the method} \label{Sec:AP_cG}
To avoid repeating the same argument for continuous finite elements
and discontinuous finite elements, we denote by $c_{ij}\upd$ either
$c_{ij}^{\textup{g,d}}$ or $c_{ij}^{\textup{b,d}}$ depending on the
context.  For any pair $i,j\in\calV$ in the same stencil, say
$j\in\calI(i)$ (or equivalently $i\in\calI(j)$), we define
\begin{subequations} \label{def_coefficients_AP}
  \begin{align}
    d_{ij}^k&:=\max(\max(\bOmega_k\SCAL \bc_{ij},0),\max(\bOmega_k\SCAL \bc_{ji},0)),
              \quad \sigma_{s,ij} := \frac12 (\sigma_{s,i}+\sigma_{s,j}). \label{def_dijk_AP}\\
h_{ij} &:= 
\frac{3}{|c_{ij}\upd|}\frac{1}{\mes{\calS}}
\sum_{k\in\calL}\mu_k d_{ij}^k, 
\quad h_i := \frac{1}{\text{card}(\calI(i))-1}\sum_{j\in\calI(i){\setminus}\{i\}} h_{ij}.  \label{def_hij_AP}
\end{align}
\end{subequations} 
Notice that $d_{ij}^k=|\bc_{ij}\SCAL\bOmega_k|$ if either
$i\in\calV\upint$ or $j\in\calV\upint$ since in this case
$\bc_{ij} = -\bc_{ji}$. The stabilized formulation we consider
consists of solving the following set of linear equations:
\begin{subequations}
  \label{AP_formulation_cG_plus_BC}
\begin{align}
  \sum_{j\in\calI(i){\setminus}\{i\}}
&\frac{1}{\sigma_{s,ij}h_{ij}+1} ( \bOmega_k\SCAL\bc_{ij} -d_{ij}^k) ( \Psi_{jk} - \Psi_{ik})
  +m_i \sigma_{a,i}\Psi_{ik}   \label{AP_formulation_cG}\\
  &\qquad =  m_i q_{ik}+ \frac{m_i\sigma_{s,i}}{\sigma_{s,i}h_i+1}\left(-\Psi_{ik}+\oPsi_{i}\right)
+ \frac{1}{\sigma_{s,i}h_i+1}b_{ik}^\partial (\beta_{ik}^\partial-\Psi_{ik}). \nonumber
   \\
  \beta_{ij}^\partial 
&:= \theta_i \alpha_{ik}^\partial 
+ (1-\theta_i)(\tfrac12 \frm_i^\partial -3\bfrM_i^\partial\SCAL\bn_i),
  \qquad \theta_i :=\max(1-2\sigma_{s,i}h_i,0). 
  \end{align}%
\end{subequations}
where it is implicitly understood that $\beta_{ij}^\partial=0$ if $i\in\calV\upint$. %
\begin{remark}[Consistency]
The above formulation coincides with the centered Galerkin
approximation \eqref{galerkin_system} if $d_{ij}^k=0$.  In the general
case, \ie with $d_{ij}^k$ as defined in \eqref{def_dijk_AP}, we have
$d_{ij}^k\sim m_i h^{-1}$, where $h$ is the mesh-size; hence
$h_{ij}\sim m_i h^{-1}/(m_i h^{-2}) \sim h$ and $h_i\sim h$.  This
computation shows that both $h_{ij}$ and $h_i$ scale like the mesh-size
(at most).  Hence, \eqref{AP_formulation_cG} converges to
\eqref{dij_system} when $\sigma_s h \to 0$. In other words, the
solutions to \eqref{AP_formulation_cG} and \eqref{dij_system} are
close when the mesh-size is significantly finer than the mean free
path. The above arguments shows that \eqref{AP_formulation_cG} is a
consistent approximation of \eqref{model_pb} (the consistency error is
first-order with respect to the mesh size).
\end{remark}

\begin{remark}[Boundary conditions]
  The boundary conditions in \eqref{AP_formulation_cG_plus_BC} are
  enforced weakly.  Observe that we recover
  $\beta_{ij}^\partial\approx \alpha_{ij}^\partial$ when the boundary
  condition at the degree of freedom $i$ is isotropic, and we have
  equality $\beta_{ij}^\partial= \alpha_{ij}^\partial$ if the angular
  quadrature satisfies
  $1=\frac{4}{\mes{\calS}}\sum_{k\in\calL_i^-} \mu_k |\bOmega_k\SCAL
  \bn_i|$.
  When the boundary condition is anisotropic and when the local
  mesh-size is not small enough to resolve the mean free path, \ie
  $\sigma_{s,i} h(i)\ge 2$, we obtain
  $\beta_{ij}^\partial := \tfrac12 \frm_i^\partial
  -3\bfrM_i^\partial\SCAL \bn_i$.
  The key motivation for the proposed definition of the boundary
  condition is based on the following observation: Let
  $\psi^0:=\lim_{\epsilon\to 0} \psi^\epsilon$ where $\psi^\epsilon$
  solves the rescaled problem \eqref{model_pb_epsilon}.  Let
  $\psi_{dG,h}^\epsilon$ be the dG approximation of
  \eqref{model_pb_epsilon} with the upwind numerical flux (assuming that the
  polynomial degree is larger than or equal to $1$) and let
  $\psi_{dG}^0 := \lim_{h\to 0}\lim_{\epsilon\to 0}
  \psi_{dG,h}^{\epsilon}$;
  here the order the two limits are taken is important. Then it is
  observed in \cite[III.D]{Adams_2001} and proved in
  \cite[Th.~5.4]{Guermond_Kanschat_2010} that
  $\psi_{dG|\front}^{0}=\tfrac12 \frm^\partial -3\bfrM^\partial\SCAL
  \bn$
  (notice that all the arguments in \citep{Guermond_Kanschat_2010}
  hold true by replacing integrals over the angles by any discrete
  measure (\ie quadrature) with the properties stated in
  \eqref{angular_quadrature}). If the incoming flux at the boundary is
  such that $\tfrac12 \frm^\partial \ne 3\bfrM^\partial\SCAL \bn$, it
  is known that $\psi^0\ne \psi^0_{dG}$, but it also known
  nevertheless that $\tfrac12 \frm^\partial -3\bfrM^\partial\SCAL \bn$
  is a very good approximation of $\psi_{|\front}^0$, see \eg
  discussions in \citep[p.~318]{Adams_2001} and
  \citep[\S5.5]{Guermond_Kanschat_2010}. Moreover we have
  $\psi_{|\front}^0 = \psi_{dG,|\front}^0 = \tfrac12 \frm^\partial
  -3\bfrM^\partial\SCAL \bn = \frm^\partial$
  when $\tfrac12 \frm^\partial = 3\bfrM^\partial\SCAL \bn$ (for
  instance if the incoming flux is isotropic), see \eg
  \citep[Th.~5.3]{Guermond_Kanschat_2010}.
  \end{remark}

  \begin{remark}[Literature] Let us now show the connection
    between~\eqref{AP_formulation_cG_plus_BC} and the technique
    introduced in~\cite{Gosse_Toscani_2002}. The system solved in this
    reference is the time-dependent version of~\eqref{model_pb} in one
    space dimension with two angular directions only:
    $\rho \partial_t(u,v) +\partial_x(u,-v) + \sigma_s (u,v) =
    \sigma_s \frac12 (u+v,u+v)$.
    Using upwind finite differences (or finite volumes), the proposed
    scheme is
    $\rho \partial_t(u_i,v_i) +
    (\frac{u_i-u_{i-1}}{h},\frac{v_{i}-v_{i+1}}{h})=
    \frac{\sigma_s}{\sigma_s h +1}(v_i-u_{i-1}, u_i-v_{i+1})$;
    see Eq.~(6) therein.  After simple manipulations, we observe that
    the scheme can be recast as follows
    $\rho \partial_t(u_i,v_i) + \frac{1}{\sigma_sh +1}
    (\frac{u_i-u_{i-1}}{h},\frac{v_{i}-v_{i+1}}{h})+
    \frac{\sigma_s}{\sigma_s h +1}(u_i, v_i) =
    \frac{\sigma_s}{\sigma_s h +1}\frac12 (u_{i}+v_i,u_i+ v_{i})$.
    Hence, the trick introduced in \citep{Gosse_Toscani_2002} consists
    of multiplying both the upwind finite differences and the
    scattering terms by the coefficient $\frac{1}{\sigma_sh +1}$. This
    is exactly what is done in \eqref{AP_formulation_cG}. In our case
    the upwind finite difference is the term
    $\sum_{j\in\calI(i){\setminus}\{i\}}( \bOmega_k\SCAL\bc_{ij}
    -d_{ij}^k) ( \Psi_{jk} - \Psi_{ik})$.
    This trick is now well accepted in the finite volume literature,
    see \eg \cite[Eq.~(10)]{Buet_Cordier_2004},
    \cite[Eq.~(31)]{Buet_Despres_2006},
    \cite[Eq.~(19)]{Buet_Despres_Emmanuel_2012},
    \cite[\S2.6]{Jin_Levermore_1996}, and
    \cite[Eq. (2.4)]{Li_Wang_2017}. Notice that, in addition to our
    recasting the technique from \citep{Gosse_Toscani_2002} into a
    discretization agnostic framework, two other novelties are our
    handling of the boundary condition, which is inspired from
    \citep[III.D]{Adams_2001} and
    \citep[\S5.5]{Guermond_Kanschat_2010}, and the definitions of
    $h_{ij}$ and $h_i$; see~\eqref{def_hij_AP}.
\end{remark}

\subsection{Diffusion limit expansion}
We investigate the diffusion limit of the
formulation~\eqref{AP_formulation_cG_plus_BC} by proceeding as in
\S\ref{Sec:locking}.  We rescale the problem as in
\eqref{model_pb_epsilon} by replacing $\sigma_{s,ij}$, $\sigma_{s,i}$,
$\sigma_{a,i}$, and $q_{ik}$ by $\frac{1}{\epsilon}\sigma_{s,ij}$,
$\frac{1}{\epsilon}\sigma_{s,i}$, $\epsilon\sigma_{a,i}$, and
$\epsilon q_{ik}$, respectively. The discrete problem consists of seeking
$\bpsi_h^{\epsilon}$ such that the following holds true for all
$(i,k)\in\calV\CROSS\calL$:
\begin{multline}
  \sum_{j\in\calI(i){\setminus}\{i\}}\frac{\epsilon}{\sigma_{s,ij}h_{ij}}
\frac{1}{1+\frac{\epsilon}{\sigma_{s,ij}h_{ij}}} 
( \bOmega_k\SCAL\bc_{ij} -d_{ij}^k) ( \Psi_{jk}^\epsilon - \Psi_{ik}^\epsilon )
  +\epsilon m_i \sigma_{a,i}\Psi_{ik}^\epsilon    \\
  =  m_i q_{ik}+ \frac{m_i}{h_i}\frac{1}{1+\frac{\epsilon}{\sigma_{s,i}h_i}}
\left(-\Psi_{ik}^\epsilon +\oPsi_{i}^\epsilon \right)+  \frac{\epsilon}{\sigma_{s,i}h_{i}}
\frac{1}{1+\frac{\epsilon}{\sigma_{s,i}h_{i}}} b_{ik}^\partial (\beta_{ik}^\partial-\Psi_{ik}^\epsilon).
  \label{AP_formulation_cG_epsilon}
\end{multline}
with  $\beta_{ik}^\partial := \theta_i^\epsilon \alpha_{ik}^\partial 
+ (1-\theta_i^\epsilon)(\tfrac12 \frm_i^\partial -3\bfrM_i^\partial\SCAL\bn_i)$,
$\theta_i^\epsilon :=\max(1-2\frac{\sigma_{s,i}}{\epsilon}h_i,0)$.
\begin{theorem}[Diffusion limit] \label{Th:diffusion_limit_AP} Let
  $\bpsi_h^\epsilon$ be the solution of the linear system
  \eqref{AP_formulation_cG_epsilon}.  Assume that the mesh family
  $\famTh$ is such
$c_{ij}\upd< 0$ for all $i\in\calV, j\in\calI(i){\setminus}\{i\}$.
  Let $\bpsi_h^{0} = \lim_{\epsilon\to 0}\bpsi_h^\epsilon$.  Then
  $\bpsi_h^{0}$ is isotropic, \ie
  $\bpsi_h^{0}= (\psi^0_{h},\ldots,\psi_{h}^0)$, and for all $i\in\calV$ the scalar field
  $\psi_h^{0}:=\sum_{j\in \calV} \Psi^0_j \varphi_j$ solves 
\begin{equation}
\frac{m_i^\partial}{\sigma_{s,i}h_{i}}\delta_i^\partial\oPsi_{i}^0 
+ a(\psi_h^0,\varphi_i)
  + m_i \sigma_{a,i}\oPsi_{i}^0  
  =   m_i \oq_i+ \frac{m_i^\partial}{\sigma_{s,i}h_{i}}\delta_i^\partial \Big(\frac{\frm_i^\partial}{2}
-3 \bfrM_i^\partial\SCAL\bn_i\Big).
\label{Eq1:Th:diffusion_limit_AP}
\end{equation}
Moreover, setting
$\bJ_i^\epsilon := \frac{1}{\epsilon\mes{\calS}} \sum_{k\in\calL}
\mu_k \bOmega_k \Psi_{ik}^\epsilon$,
and $\bJ_i^0:=\lim_{\epsilon\to 0}\bJ_i^\epsilon $, the vector
$\bJ_i^\epsilon$ satisfies the following consistent approximation of
Fick's law for all $i\in\calV\upint$:
\begin{equation}
  m_i \bJ_i^0 = -\sum_{j\in\calI(i){\setminus}\{i\}} \frac{h_i}{h_{ij}} 
\frac{1}{3 \sigma_{s,ij}} \bc_{ij}(\Psi^0_j-\Psi^0_i).
  \label{Eq2:Th:diffusion_limit_AP}
\end{equation}
\end{theorem}
\begin{proof}%
  A rigorous functional analytic argument can be made by proceeding as
  in \citep[\S4]{Guermond_Kanschat_2010}, but since the mesh-size is
  fixed and the approximation space is finite-dimensional, there is no
  fundamental obstacle to proceed formally; hence, we consider the
  asymptotic expansion
  $\bpsi_h=\bpsi_h^0+\epsilon\bpsi_h^1+\epsilon^2\bpsi_h^2+\calO(\epsilon^3)$.

  Proof of~\eqref{Eq1:Th:diffusion_limit_AP}. Notice first that
  $\theta_i^\epsilon=0$ for all $\epsilon \le 2\sigma_{s,i} h_i$;
  hence,
  $\beta_{ik}^\partial =\beta_{i}^\partial := \tfrac12 \frm_i^\partial
  -3\bfrM_i^\partial\SCAL\bn_i$.
  Using that
  $\frac{1}{1+\frac{\epsilon}{\sigma h}} = 1 - \frac{\epsilon}{\sigma
    h} + \calO(\epsilon^2)$, we have
\begin{multline*}
  \sum_{j\in\calI(i){\setminus}\{i\}}\frac{\epsilon}{\sigma_{s,ij}h_{ij}}
( \bOmega_k\SCAL\bc_{ij} -d_{ij}^k) ( \Psi_{jk} - \Psi_{ik})
  +\epsilon m_i \sigma_{a,i}\Psi_{ik}   \\
  =  \epsilon m_i q_{ik}+ \frac{m_i}{h_i}(1-\frac{\epsilon}{\sigma_{s,i}h_i})\left(-\Psi_{ik}+\oPsi_{i}\right) 
+\frac{\epsilon}{\sigma_{s,i}h_{i}}b_{ik}^\partial (\beta_{i}^\partial-\Psi_{ik})
+ \calO(\epsilon^2).
\end{multline*}
Inserting now the formal asymptotic expansion
$\bpsi_h=\bpsi_h^0+\epsilon\bpsi_h^1+\calO(\epsilon^2)$ into this equation,
we infer that $\oPsi_{i}^0 - \Psi_{ik}^0 =0$ for all $(i,k)\in\calV\CROSS\calL$ and 
\begin{multline}
\sum_{j\in\calI(i){\setminus}\{i\}}\frac{1}{\sigma_{s,ij}h_{ij}}
( \bOmega_k\SCAL\bc_{ij} -d_{ij}^k) ( \Psi_{jk}^0 - \Psi_{ik}^0)
  + m_i \sigma_{a,i}\Psi_{ik}^0 \\ 
  =   m_i q_{ik}+ \frac{m_i}{h_i} \left(-\Psi_{ik}^1+\oPsi_{i}^1\right)
+\frac{1}{\sigma_{s,i}h_{i}} b_{ik}^\partial (\beta_{i}^\partial-\Psi_{ik}^0).
  \label{corr1:Th:diffusion_limit_AP}
\end{multline}
Taking the (weighted) average of the second equation over the
discrete ordinates, we obtain
\[
\sum_{j\in\calI(i){\setminus}\{i\}}( \oPsi_{j}^0 - \oPsi_{i}^0)\frac{1}{\sigma_{s,ij}h_{ij}}
 \frac{1}{\mes{\calS}} \sum_{k\in\calL} -\mu_k d_{ij}^k
  + m_i \sigma_{a,i}\oPsi_{i}^0  
  =   m_i \oq_{i} +  \frac{m_i^\partial}{\sigma_{s,i}h_{i}}(\delta_i^\partial\beta_i^\partial - \delta_i^\partial\oPsi_{i}^0).
\]
(Recall that $\delta_i^\partial \approx \frac14$).  Now we use the
definition of $h_{ij}$ (see \eqref{def_hij_AP}) and recall that the
mesh family $\famTh$ is assumed to be such that
$c_{ij}\upd < 0$ for all $i\in\calV$, $j\in\calI(i){\setminus}\{i\}$; then we obtain
\[
  \frac{m_i^\partial}{\sigma_{s,i}h_{i}}\delta_i^\partial\oPsi_{i}^0
  + \!\! \sum_{j\in\calI(i){\setminus}\{i\}}\!\!\!
 \frac13 c_{ij}\upd ( \oPsi_{j}^0 - \oPsi_{i}^0) 
  + m_i \sigma_{a,i}\oPsi_{i}^0  
  =   m_i \oq_{i}+ \frac{m_i^\partial}{\sigma_{s,i}h_{i}}\delta_i^\partial \beta_i^\partial.
\]
Now using the partition of unity property, \ie
$\sum_{j\in\calI(i)}c_{ij}\upd =0$,
and recalling the definition of $\beta_i^\partial$, we infer that
\[
\frac{m_i^\partial}{\sigma_{s,i}h_{i}}\delta_i^\partial\oPsi_{i}^0 
+ a(\GRAD\psi_h^0,\varphi_i)
  + m_i \sigma_{a,i}\oPsi_{i}^0  
  =   m_i \oq_i+ \frac{m_i^\partial}{\sigma_{s,i}h_{i}}\delta_i^\partial (\frac12 \frm_i^\partial
-3 \bfrM_i^\partial\SCAL\bn_i).
\]

Proof of~\eqref{Eq2:Th:diffusion_limit_AP}. Since $\bpsi_h^0$ is isotropic, we have
\[
  \bJ_i^\epsilon := \frac{1}{\epsilon\mes{\calS}} \sum_{k\in\calL}
  \mu_k \bOmega_k \Psi_{ik}^\epsilon = \frac{1}{\mes{\calS}} \sum_{k\in\calL}
  \mu_k \bOmega_k \Psi_{ik}^1 + \calO(\epsilon).
\]
That is,
$\bJ_i^0:=\lim_{\epsilon\to 0}\bJ_i^\epsilon = \frac{1}{\mes{\calS}}
\sum_{k\in\calL} \mu_k \bOmega_k \Psi_{ik}^1$.
We now multiply \eqref{corr1:Th:diffusion_limit_AP} by $\bOmega_k$,
take the (weighted) average over the discrete ordinates, and recall that the
angular quadrature satisfies
$\sum_{k\in\calL}\mu_k \bOmega_k |\bn \SCAL \bOmega_k|=\bzero$ for all
$\bn\in \Real^3$,
\begin{multline*}
\sum_{j\in\calI(i){\setminus}\{i\}}\!\!\frac{1}{3\sigma_{s,ij}h_{ij}}
 \bc_{ij} ( \Psi_{j}^0 - \Psi_{i}^0)
-\sum_{j\in\calI(i){\setminus}\{i\}}\!\!\frac{( \Psi_{j}^0 - \Psi_{i}^0)}{\sigma_{s,ij}h_{ij}} 
\sum_{k\in\calL}\frac{\mu_k}{\mes{\calS}} \bOmega_k d_{ij}^k \\
=   -\frac{m_i}{h_i} \bJ_i^0 + \frac{m_i^\partial}{\sigma_{s,i}h_{i}}\frac16 (\beta_i^\partial
- \Psi_i^0)\bn_i,
\end{multline*}
where we used that
$\frac{1}{\mes{\calS}}\sum_{k\in \calL_i^{-}} |\bOmega_k\SCAL \bc|
\bOmega_k = \frac16 \bc$
for any $\bc\in \Real^3$.  If $i\in\calV\upint$, then
$d_{ij}^k = |\bc_{ij}\SCAL \bOmega_k|$, which in turn implies that
$\sum_{k\in\calL}\mu_k \bOmega_k d_{ij}^k =0$. The assertion follows
readily.
\end{proof}

\begin{remark}[Limit problem and boundary conditions]
  Since $h_i$ behaves like the mesh size, $h$, the discrete
  problem~\eqref{Eq1:Th:diffusion_limit_AP} is a weak formulation with
  a penalty on the boundary condition scaling like $h^{-1}$.  The
  continuous problem associated with the discrete
  problem~\eqref{Eq1:Th:diffusion_limit_AP} consists of seeking
  $\psi^{\textup{lim}}\in H^1(\Dom)$ so that
  $-\DIV\Big(\frac{1}{3\sigma_s}\GRAD \psi^{\textup{lim}}\Big) + \sigma_a \psi^{\textup{lim}} =
  \oq$,
  with
  $\psi^{\textup{lim}}_{|\front} = \frac12\frm^\partial -3 \bfrM^\partial\SCAL
  \bn$.
  This result is coherent with
  \citep[Thm.~5.4]{Guermond_Kanschat_2010}. Recall that in general $\psi^{\textup{lim}}\ne \psi^0$ unless
$\frac12\frm^\partial +3 \bfrM^\partial\SCAL
  \bn =0$, see \citep[\S5.5]{Guermond_Kanschat_2010}.
\end{remark}

\begin{remark}[Fick's law]
Let us now interpret~\eqref{Eq2:Th:diffusion_limit_AP}. Assume that
the mesh is uniform or quasi-uniform in the neighborhood of the
Lagrange node $\ba_i$, then $h_i\approx h_{ij}$ and
$\sigma_{s,ij}\approx \sigma_{s,i}$. Hence,
$m_i \bJ_i^0\approx - \frac{1}{3\sigma_{s,i}}\sum_{j\in\calI(i)}
\bc_{ij} \Psi_j^0$.
Owing to the definition of the coefficients $\bc_{ij}$, this equation is a
consistent approximation of Fick's law
$\bJ = -\frac{1}{3\sigma_s} \GRAD \psi$.
\end{remark}

\begin{remark}[Meshes] 
  It is known for simplicial meshes and piecewise linear
  continuous finite elements that a sufficient condition for the
  inequality
  $c_{ij}^{\textup{g,d}} <0$
  to hold for all $i\in\calV$, $j\in\calI(i){\setminus}\{i\}$ is that the mesh family
  $\famTh$ satisfies the so-called acute angle condition, \eg
  \cite[Eq.~(2.5)]{XuZik:99}.
\end{remark}

\subsection{Positivity}
We establish in this section the positivity of the method defined
in~\eqref{AP_formulation_cG_plus_BC} using the definitions
in~\eqref{def_coefficients_AP}. We set $\Psi^{\min}:=\min_{(j,l)\in\calV\CROSS\calL} \Psi_{j,l}$
and $\Psi^{\max}:=\max_{(j,l)\in\calV\CROSS\calL} \Psi_{j,l}$.
\begin{theorem}[Minimum/Maximum principle] \label{Th:AP_positivity}
Let   $(\Psi_{ik})_{(i,k)\in\calV\CROSS \calL}$ be the solution to  \eqref{AP_formulation_cG_plus_BC}
with $d_{ij}^k$ and all the other parameters defined in~\eqref{def_dijk_AP}-\eqref{def_hij_AP}.
Let   $(i_0,k_0)$, $(i_1,k_1)\in \calV\CROSS \calL$ be such that $\Psi_{i_0k_0} =\Psi^{\min}$ 
and   $\Psi_{i_1k_1} = \Psi^{\max}$.
\begin{enumerate}[(i)]
\item \label{item1:Thm:min_principle} Assume that
  $\min_{(j,l)\in\calV\CROSS\calL}
  (\sigma_{a,j}+b_{jl}^\partial)>0$. Then
\begin{equation}
  \tfrac{m_{i_0} q_{i_0k_0} + \frac{b_{i_0k_0}^\partial}{\sigma_{s,i_0} h_{i_0} +1} \beta_{i_0k_0}^\partial}{m_{i_0}
    \sigma_{a,i_0} + \frac{b_{i_0k_0}^\partial}{\sigma_{s,i_0} h_{i_0} +1}} \le \Psi^{\min}\le 
  \Psi^{\max}\le  \tfrac{m_{i_1} q_{i_1k_1} + \frac{b_{i_1k_1}^\partial}{\sigma_{s,i_1} h_{i_1} +1} \beta_{i_1k_1}^\partial
  }{m_{i_1}
    \sigma_{a,i_1} +  \frac{b_{i_1k_1}^\partial}{\sigma_{s,i_1} h_{i_1} +1} }.
  \quad 
\end{equation}
\item \label{item2:Thm:min_principle} Otherwise, assume that for all
  $i\in\calV$ such that $\sigma_{a,i}=0$ and $b_{ik}^\partial=0$ the
  definition of $d_{ij}^k$ is slightly modified so that
  $\bOmega_k\SCAL\bc_{ij} < d_{ij}^k$ for all $j\in\calI(i)$ (instead
  of $\bOmega_k\SCAL\bc_{ij} \le d_{ij}^k$). If
  $0\le \min_{(i,k)\in \calV\CROSS\calL} q_{ik}$ and
  $0\le \min_{(i,k)\in (\calV\CROSS\calL)^\partial}
  \alpha_{ik}^\partial$, then $0\le \Psi^{\min}$.
\item \label{item3:Thm:min_principle} Moreover, under the same
  assumptions on $d_{ij}^k$ as in \eqref{item2:Thm:min_principle}, if
  $\max_{(i,k)\in \calV\CROSS\calL} q_{ik}\le 0$, then
  $\Psi^{\max} \le \max_{(i,k)\in (\calV\CROSS\calL)^\partial}
   \beta_{ik}^\partial$
\end{enumerate}
\end{theorem}
\begin{proof} We proceed as in the proof of Lemma~\ref{Lem:min_principle}.
 We start with the proof of~\eqref{item1:Thm:min_principle} and
 assume that
  $\min_{j\in\calV} (\sigma_{t,j}- \sigma_{s,j})>0$. Let
  $(i_0,k_0)\in \calV\CROSS \calL$ be the indices of the degree of freedom
  where the minimum is attained; that is, $\Psi_{ik}\ge \Psi_{i_0k_0}$
  for all $(i,k)\in\calV\CROSS \calL$. Then using that
  $\bOmega_k\SCAL\bc_{ij} - d_{ij}^k \le \max(\bOmega_k\SCAL\bc_{ij},0) -
  d_{ij}^k  \le 0$,
together with $\Psi_{jk_0} - \Psi_{i_0k_0} \ge 0$ for all $j\in\calI(i_0)$, and
 $\Psi_{i_0k_0} \le \oPsi_{i_0} $, we infer that
\begin{align*}
   m_{i_0} q_{i_0k_0} 
& + \frac{b_{i_0k_0}^\partial}{\sigma_{s,i_0} h_{i_0} +1}\beta_{i_0k_0}^\partial 
= \!\!\sum_{j\in\calI(i_0){\setminus}\{i_0\}}\!\!
   \frac{\bOmega_{k_0}\SCAL\bc_{i_0j} - d_{i_0j}^{k_0}}{\sigma_{s,i_0j}h_{i_0j}+1} (\Psi_{jk_0} - \Psi_{i_0k_0})  \\
& + \frac{m_{i_0} \sigma_{s,i_0}}{\sigma_{s,i_0} h_{i_0} +1} (\Psi_{i_0k_0} -\oPsi_{i_0})
+ m_{i_0} \sigma_{a,i_0} \Psi_{i_0k_0} + \frac{b_{i_0k_0}^\partial}{\sigma_{s,i_0} h_{i_0} +1}\Psi_{i_0k_0}  \\
  & \le  m_{i_0} \sigma_{a,i_0} \Psi_{i_0k_0} + \frac{b_{i_0k_0}^\partial}{\sigma_{s,i_0} h_{i_0} +1}\Psi_{i_0k_0}.
 \end{align*}
 Hence
 $m_{i_0} q_{i_0k_0} + \frac{b_{i_0k_0}^\partial}{\sigma_{s,i_0} h_{i_0}
   +1} \beta_{i_0k_0}^\partial \le m_{i_0} \sigma_{a,i_0} +
 \frac{b_{i_0k_0}^\partial}{\sigma_{s,i_0} h_{i_0} +1}\Psi_{i_0k_0}$.
 The assertion follows readily. The proof of the other assertion regarding
 $\Psi^{\max}$ is analogous.

 Proof of \eqref{item2:Thm:min_principle} assuming that
 $0\le \min_{(i,k)\in \calV\CROSS\calL} q_{ik}$ and
 $0\le \min_{(i,k)\in (\calV\CROSS\calL)^\partial}
 \alpha_{ik}^\partial$.
 From part~\eqref{item1:Lem:min_principle} we conclude that we need to
 prove $\Psi_{i_0k_0}\ge 0$ only in the case $\sigma_{a,i_0} =0 $ and
 $b_{i_0k_0}^\partial=0$. Assuming that $\sigma_{a,i_0} =0 $ and
 $b_{i_0k_0}^\partial=0$, we have from
 part~\eqref{item1:Lem:min_principle} the following inequality
 \[
   0 \le m_{i_0} q_{i_0k_0} 
      \le \!\!\sum_{j\in\calI(i_0){\setminus}\{i_0\}}\!\!
   \frac{\bOmega_{k_0}\SCAL\bc_{i_0j} - d_{i_0j}^{k_0}}{\sigma_{s,i_0j_0}h_{i_0j}+1} 
(\Psi_{jk_0} - \Psi_{i_0k_0}) \le 0.
 \]
 The assumption $\bOmega_k\SCAL\bc_{i_0j} - d_{i_0j}^{k_0}<0$ for all
 $j\in\calI(i_0)$, implies that $\Psi_{jk_0} - \Psi_{i_0k_0}=0$ for
 all $j\in\calI(i_0)$. Therefore, we conclude that the global minimum
 is attained not only at the degree of freedom $(i_0,k_0)$ but also in
 the whole neighborhood, \ie for all $j\in\calI(i_0)$. Repeating the
 above argument for a global minimum at $(j, k_0)$ for all
 $j\in\calI(i_0)$, we derive that the global minimum is either
 nonnegative (if $m_{j} \sigma_{a,j} + b_{jk_0}^\partial>0$) or again
 attained in the whole neighborhood of $j$, \ie for all
 $s\in\calI(j)$. This process can terminate in two ways: (i) either
 the global minimum is nonnegative at some $j$, \ie
 $m_{j} \sigma_{a,j} + b_{jk_0}^\partial>0$; 
 (ii) or the global
 minimum is attained at all of the degrees of freedom topologically
 connected to $i_0$.  In the first case we have proved the
 non-negativity, in the second case we have that
 $\Psi_{jk_0}= \Psi_{i_0k_0}$ for all $j$ in the same connected
 component as $i_0$, which is the entire set $\calV$ since $\calT_h$
 is connected (because $\Dom$ is connected).  However, for any fixed
 $k_0$ there exists $j$ such that $\Psi_{jk_0}$ is on the inflow
 boundary for $\bOmega_{k_0}$. That is, we have  $b_{jk_0}^\partial>0$,
 and  conclude (see \eqref{AP_formulation_cG_plus_BC}) that
  $\Psi_{i_0k_0}\ge 0$ because $\beta_{ij}^\partial = \theta_i \alpha_{ik}^\partial 
+ (1-\theta_i)(\tfrac12 \frm_i^\partial -3\bfrM_i^\partial\SCAL\bn_i)\ge 0$
on the  the inflow boundary.

 Proof of \eqref{item3:Thm:min_principle} assuming that
 $\min_{(i,k)\in \calV\CROSS\calL} q_{ik}\le 0$. By
 proceeding as in Step~\eqref{item1:Lem:min_principle}, we infer that
 \begin{align*}
   m_{i_1} q_{i_1k_1} & 
+ \frac{b_{i_1k_1}^\partial}{\sigma_{s,i_1} h_{i_1} +1} (\beta_{i_1k_1}^\partial-\Psi_{i_1k_1}) \\
  & \ge  \!\!\sum_{j\in\calI(i_1){\setminus}\{i_1\}}\!\!
   \frac{\bOmega_{k_1}\SCAL\bc_{i_1j} - d_{i_1j}^{k_1}}{\sigma_{s,i_1j}h_{i_1j}+1} (\Psi_{jk_1} - \Psi_{i_1k_1})
  + m_{i_1} \sigma_{a,i_1}\Psi_{i_1k_1}  \ge 0,
 \end{align*}
 \ie
 $(m_{i_1} \sigma_{a,i_1} + \frac{b_{i_1k_1}^\partial}{\sigma_{s,i_1} h_{i_1} +1} )\Psi_{i_1k_1} \le
 m_{i_1} q_{i_1k_1} + \frac{b_{i_1k_1}^\partial}{\sigma_{s,i_1} h_{i_1} +1} \beta_{i_1k_1}^\partial$;
 which implies $\Psi_{i_1k_1} \le \beta_{i_1k_1}^\partial$ if $b_{i_1k_1}^\partial>0$. Hence we just
 need to consider the case $b_{i_1k_1}^\partial=0$. In that case
 $0 \ge \sum_{j\in\calI(i_1){\setminus}\{i_1\}}
 \frac{\bOmega_{k_1}\SCAL\bc_{i_1j} - d_{i_1j}^{k_1}}{\sigma_{s,i_1j}h_{i_1j}+1}  (\Psi_{jk_1} -
 \Psi_{i_1k_1}) \ge 0$ and $\Psi_{jk_1} =\Psi_{i_1k_1}$ for all
 $j\in \calI(i_1)$. Then we proceed as in
 Step~\eqref{item2:Lem:min_principle} until we reach a degree of freedom $j$ that is on
 the inflow boundary for $\bOmega_{k_1}$, \ie $b_{jk_1}^\partial>0$.
 Then $\Psi^{\max} = \Psi_{i_1k_1} = \Psi_{jk_1} \le \beta_{j,k_1}^\partial$.
 \end{proof}

\section{Numerical illustrations} \label{Sec:Numerical_results} 
We present in this section numerical results to illustrate the
positive and asymptotic preserving
algorithm~\eqref{AP_formulation_cG_plus_BC} described in
\S\ref{Sec:AP_cG}.   We compare this
technique in various regimes with the standard dG1 technique using the upwind flux.
 
\subsection{Numerical details} \label{Sec:Numerical_details}
The positive and asymptotic preserving algorithm defined
in~\eqref{AP_formulation_cG_plus_BC} is implemented with piecewise
linear continuous finite elements on simplices. We use the same code
for one-dimensional and two-dimensional tests. The meshes in one
dimension are uniform. The meshes in two space dimension are
non-uniform, composed of triangles, and have the Delaunay
property. Nothing special is done to make the triangulations satisfy
the acute angle condition, \ie the condition may not be satisfied for
a few pairs of vertices.  In one dimension we use the Gauss-Legendre
quadrature for the angular discretization: the $x_1$-component of the
angles are the quadrature points of the Gaussian quadrature over
$[-1,1]$ and the weights are the weights of the Gaussian quadrature.
In two-dimensions we use the standard triangular $S_N$
quadrature (Gauss-Legendre quadrature along the polar axis and
equi-distributed angles along the azimuth with $\frac18 N(N+2)$ angles per octant).  Since the size of the
problems involved here is small (at most $2\CROSS 10^6$ degrees of
freedom), we assemble the sparse matrix defined in
\eqref{AP_formulation_cG_plus_BC} using the compressed sparse row
format and solve it using Pardiso~(see \eg \citet{Pardiso_2014}).
More sophisticated techniques involving source iterations and
synthetic acceleration could be used for significantly larger
systems. We do not discuss this issue since it is out of the scope of
the paper.

In order to assess the asymptotic-preserving approach, we compare it
against a state-of-the-art technique. More specifically,
\eqref{model_pb} is solved using dG1 with the upwind numerical flux
and the same triangular $S_N$ quadrature as above. The linear system
is solved by iterating on the scattering source (see \eg
\cite{adams_larsen_iter_methods}); for instance starting with some
guess $\obpsi_{h}^{(0)}$, one constructs a sequence
$\bpsi_h^{(0)},\ldots, \bpsi_{h}^{(\ell)},\ldots$ Given some state
$\bpsi_{h}^{(\ell)}$ we compute an intermediate state
$\bpsi_{h}^{(\ell+\frac12)}$ such that
\begin{subequations}
\begin{align}
 \sum_{j\in\calI(i)}\!\! A_{ij}^k \Psi_{jk}^{(\ell+\frac12)}
+ m_i \sigma_{t,i} \Psi_{ik}^{(\ell+\frac12)}
 +b_{ik}^\partial \Psi_{ik}^{(\ell+\frac12)} = m_i \sigma_{s,i} \oPsi_{i}^{(\ell)} + m_i q_{ik}
+b_{ik}^\partial \alpha_{ik}^\partial, \label{dG_system_upwind} \\
A_{ij}^k := 
  \begin{cases} 
\int_K  (\bOmega_k\SCAL \GRAD\varphi_j)\varphi_i\diff \bx 
& \text{$j\in \calI(K){\setminus}\calI(\partial K\upi)$} \\
\int_K  (\bOmega_k\SCAL \GRAD\varphi_j)\varphi_i \diff \bx 
+\int_{\partial K}\varphi_i \varphi_j (\bOmega_k\SCAL\bn_K)_{-}
 \diff \bx,
& \text{$j\in \calI(\partial K\upi)$} \\
-\int_{\partial_K} \frac{|\bOmega_k\SCAL\bn_K|-\bOmega_k\SCAL\bn_K}{2} \varphi_i \varphi_j \diff \bx,
& \text{$j\in \calI(\partial K\upe)$},
\end{cases}
\end{align}
\end{subequations}
with $z_{-}:=\frac12(|z|-z)$.
For each direction $k$, \eqref{dG_system_upwind} is solved
cell-by-cell by sweeping through the mesh from the inflow boundary to
the outflow boundary defined by the angle $\bOmega_k$ (a process
termed ``transport sweep'' in the radiation transport
community). 
Without synthetic acceleration, we
set $\bpsi_h^{(\ell+1)} = \bpsi_h^{(\ell+1/2)}$ and the new source
iteration ($\ell \leftarrow \ell+1$) can proceed.  However, in highly
diffusive configurations, a diffusion synthetic accelerator is invoked
to compute a correction $\delta \bpsi_h^{\ell+1}$ to improve the scalar
flux iterate; at the end of the process we set
$\bpsi_h^{(\ell+1)} = \bpsi_h^{(\ell+1/2)} + \delta \bpsi_h^{\ell+1}.$ Here, we use a
dG compatible diffusion synthetic accelerator based on an interior
penalty technique; see \eg \cite{wang_ragusa_dsa}  
for additional details. 

\subsection{Manufactured solution}
We first test our piecewise linear, continuous finite element
implementation of the algorithm described in \S\ref{Sec:AP_cG} on a
manufactured solution.  The domain is $\Dom=(0,1)^2\CROSS\Real$, with
$\sigma_t=\sigma_s=1$, and the solution is
$\bpsi:=(\psi_1,\ldots,\psi_L)$ with
\begin{equation}
  \psi_{k}(\bx) = 2+\sin(\bOmega_k\SCAL\bx)+ \sin(\pi x_1)\sin(\pi x_2), \label{manufactured-solution}
\end{equation}
where $k\in \calL$, $\bx:=(x_1,x_2)\in\Dom$.
The source term $q(\bx,\bOmega_k)$ is computed accordingly with
$\opsi(\bx):= \frac{1}{\mes{\calS}} \sum_{k\in\calK} \psi_{k}(\bx)$.

The relative errors in the $L^2$-norm, $L^\infty$-norm, and
$H^1$-semi-norm are calculated on five nonuniform meshes composed of
triangles with $140$, $507$, $1927$, $7545$, and $29870$ Lagrange
nodes, respectively; the corresponding mesh-sizes are approximately
$h\approx 0.1$, $0.5$, $0.025$, $0.125$, and $0.00625$. We define the
error $\be:=(e_1,\ldots, e_L)$ with
$e_k:=\psi_{h,k}-\Pi_h\upL(\psi_k)$, where $\Pi_h\upL(\psi_k)$ is the
Lagrange interpolant of $\psi_k$ in $P\upg(\calT_h)$, and we set
\begin{equation}
  \|\be\|_{L^2}^2 = \sum_{k\in\calL} \mu_k \|e_k\|_{L^2(\Dom)}^2, \quad
  \|\be\|_{L^\infty}= \max_{k\in\calL} \|e_k\|_{L^\infty(\Dom)},
\end{equation}
The relative errors are denoted and defined as follows:
$\text{rel}(\|\be\|_{L^2})= \|\be\|_{L^2}/\|\psi\|_{L^2}$,
$\text{rel}(\|\be\|_{L^\infty})= \|\be\|_{L^\infty}/\|\psi\|_{L^\infty}$,
$\text{rel}(\|\GRAD \be\|_{\bL^2})= \|\GRAD \be\|_{\bL^2}/\|\GRAD
\psi\|_{\bL^2}$. The results for the $S_6$ and $S_{10}$ quadratures
are reported in
Table~\ref{Table:Convergence_2D_manufactured_solution}. We observe
that, as expected, the method is first-order accurate in space in the
$L^2$-norm, and it is $\calO(h^{\frac12})$ in the $L^\infty$-norm and
in the $H^1$-semi-norm. These results are compatible with the best
theoretical error estimates known for the approximation of the linear
transport equation using first-order viscosities. 
\begin{table}[h]
\setlength{\tabcolsep}{3pt}
\centering\scriptsize
\begin{tabular}{|r|c|c|c|c|c|c|c|} \hline
 & \#dofs & $\text{rel}(\|\be\|_{L^2}\!)$ & rate & $\text{rel}(\|\be\|_{L^\infty}\!)$ & rate & $\text{rel}(\|\GRAD \be\|_{\bL^2}\!)$ & rate\\ \hline
  \multirow{5}{*}{\rotatebox[origin=l]{90}{$S_6$}}
&   140 & 5.20E-02 & --  & 2.89E-01 & -- & 3.07E-01 & -- \\ \cline{2-8}
&  507 & 2.70E-02 & 1.02 & 2.08E-01 & 0.51 & 2.01E-01 & 0.66 \\ \cline{2-8}
& 1927 & 1.37E-02 & 1.01 & 1.48E-01 & 0.51 & 1.36E-01 & 0.59 \\ \cline{2-8}
& 7545 & 6.93E-03 & 1.00 & 1.05E-01 & 0.50 & 9.38E-02 & 0.54 \\ \cline{2-8}
 & 29870 & 3.48E-03 & 1.00 & 7.48E-02 & 0.50 & 6.55E-02 & 0.52 \\ \hline
  \multirow{5}{*}{\rotatebox[origin=l]{90}{$S_{10}$}}
 &   140 & 5.19E-02 & --  & 2.91E-01 & -- &3.07E-01& -- \\ \cline{2-8}
 & 507 & 2.69E-02 & 1.02 & 2.08E-01 & 0.52 & 2.01E-01 & 0.66 \\ \cline{2-8}
 &1927 & 1.37E-02 & 1.01 & 1.48E-01 & 0.51 & 1.37E-01 & 0.58 \\ \cline{2-8}
 &7545 & 6.93E-03 & 1.00 & 1.09E-01 & 0.45 & 9.48E-02 & 0.54 \\ \cline{2-8}
&29870 & 3.48E-03 & 1.00 & 8.22E-02 & 0.42 & 6.64E-02 & 0.52 \\ \hline
\end{tabular}%
\caption{Convergence tests with 
respect to mesh-size  with solution~\eqref{manufactured-solution} and quadrature $S_6$ and $S_{10}$.}%
\label{Table:Convergence_2D_manufactured_solution}%
\end{table}

\subsubsection{Diffusion limit with constant cross sections}
We consider the two-dimensional domain $\Dom=(0,1)^2\CROSS \Real$ with
constant cross sections $\sigma_t=\sigma_s=\frac{1}{\epsilon}$ and
source term
$q(\bx)=\epsilon \frac{2}{3}\pi^2\sin(\pi x_1)\sin(\pi x_2)$. The
diffusion limit corresponding to $\epsilon\to 0$ is
$\psi^0(\bx) = \sin(\pi x_1)\sin(\pi x_2)$. We solve \eqref{model_pb}
with continuous linear finite elements and the algorithm described in
\S\ref{Sec:AP_cG}.  The meshes are nonuniform and composed of
triangles. To estimate the convergence we use five meshes with $140$,
$507$, $1927$, $7545$, and $29870$ Lagrange nodes, respectively; the
corresponding mesh-sizes are approximately $h\approx 0.1$, $0.5$,
$0.025$, $0.125$, and $0.00625$. We use the $S_6$ 
angular quadrature.

\begin{table}[h]
\setlength{\tabcolsep}{3pt}
\centering\scriptsize
\begin{tabular}{|r|c|c|c|c|c|} \hline
$\epsilon$ & \#dofs & $\text{rel}(\|\be\|_{L^2}\!)$ & rate & $\text{rel}(\|\GRAD \be\|_{\bL^2}\!)$ & rate \\ \hline
\multirow{5}{*}{\rotatebox[origin=l]{90}{$10^{-3}$}}
& 140 & 2.01E-02 & --  & 9.68E-03 & --   \\ \cline{2-6}
&  507 & 2.15E-03 & 2.34 & 8.00E-03 & 1.44 \\ \cline{2-6}
& 1927 & 2.91E-03 & -.45 & 6.62E-03 & 0.28 \\ \cline{2-6}
& 7545 & 3.11E-03 & -.10 & 7.75E-03 & -.23 \\ \cline{2-6}
&29870 & 3.17E-03 & -.03 & 8.84E-03 & -.19 \\ \hline 
\multirow{ 5}{*}{\rotatebox[origin=l]{90}{$10^{-4}$}}
&  140 & 1.92E-02 & --  & 1.20E-02 & --   \\ \cline{2-6}
&  507 & 2.87E-03 & 2.22 & 5.85E-03 & 1.85 \\ \cline{2-6}
& 1927 & 5.43E-04 & 2.49 & 2.16E-03 & 1.49 \\ \cline{2-6}
& 7545 & 2.01E-04 & 1.45 & 1.21E-03 & 0.85 \\ \cline{2-6}
&29870 & 2.53E-04 & -.33 & 1.33E-03 & -.13 \\ \hline
\end{tabular}\hfill \begin{tabular}{|r|c|c|c|c|c|} \hline
$\epsilon$ & \#dofs & $\text{rel}(\|\be\|_{L^2}\!)$ & rate & $\text{rel}(\|\GRAD \be\|_{\bL^2}\!)$ & rate \\ \hline
\multirow{ 5}{*}{\rotatebox[origin=l]{90}{$10^{-5}$}}
&  140 & 1.92E-02 & --  & 1.22E-02 & --   \\ \cline{2-6}
&  507 & 3.12E-03 & 2.12 & 5.76E-03 & 1.87 \\ \cline{2-6}
& 1927 & 7.59E-04 & 2.12 & 1.99E-03 & 1.59 \\ \cline{2-6}
& 7545 & 1.72E-04 & 2.18 & 7.17E-04 & 1.49 \\ \cline{2-6}
&29870 & 3.28E-05 & 2.41 & 2.73E-04 & 1.40 \\ \hline
\multirow{ 5}{*}{\rotatebox[origin=l]{90}{$10^{-6}$}}
&  140 & 1.91E-02 & --   & 1.22E-02 & --   \\ \cline{2-6}
&  507 & 3.14E-03 & 2.11 & 5.75E-03 & 1.87 \\ \cline{2-6}
& 1927 & 7.84E-04 & 2.08 & 1.98E-03 & 1.60 \\ \cline{2-6}
& 7545 & 1.93E-04 & 2.06 & 7.07E-04 & 1.51 \\ \cline{2-6}
&29870 & 4.64E-05 & 2.07 & 2.35E-04 & 1.60 \\ \hline
\end{tabular}
\caption{Convergence test on $\be:=\obpsi_h-\Pi_h\upL(\bpsi^0)$ with 
respect to the mesh-size and $\epsilon$.}
\label{Table:Convergence_2D_homogeneous}
\end{table}
The results for $\epsilon\in\{10^{-3},10^{-4},10^{-5},10^{-6}\}$ are
reported in Table~\ref{Table:Convergence_2D_homogeneous}. We show in
this table the relative $L^2$-norm and the relative $H^1$-semi-norm of the difference
$\obpsi_h-\Pi_h\upL(\bpsi^0)$, where $\Pi_h\upL(\bpsi^0)$ is the Lagrange
interpolant of $\bpsi^0$. We clearly observe that, just like proved  in
\citep[Th.~5.3]{Guermond_Kanschat_2010} for the upwind dG1 approximation, the scalar flux $\obpsi_h$
converges optimally to $\bpsi_h^0$ when $\epsilon$ is significantly
smaller than the mesh-size. The convergence order is $\calO(h^2)$ in
the $L^2$-norm.  It seems that some super-closeness phenomenon occurs
in the $H^1$-semi-norm since
$\|\GRAD(\obpsi_h-\Pi_h\upL(\bpsi^0))\|_{\bL^2}$ converges like
$\calO(h^{1.5})$.

\subsection{One-dimensional results} \label{Sec:oneD}
We now perform four one-dimensional tests and compare the positive
asymptotic preserving method (with piecewise linear continuous finite
elements) with the upwind dG1 approximation. We use an $S_8$ 
angular quadrature (8 discrete directions in 1D) for all the cases. The
angles are enumerated in increasing order from 1 to 8. The data for
the four cases are reported in Table~\ref{Tab:oneD_data}.  The
boundary condition for cases 1, 3, and 4 are $\bpsi_{h|\front_-}=0$
(this is the so-called vacuum boundary condition).  The boundary
conditions for case 2 are $\bpsi_{h,k}=0$ for $k\ne 5$, $1\le k\le 8$,
and $\psi_{h,5}(0)=1.0$.
\begin{table}[h]
\setlength{\tabcolsep}{3pt}
\centering\scriptsize
\begin{tabular}{|r|c|c|c|c|c|c|} \hline
  & \#zones & \multicolumn{5}{|c|}{5}  \\ \hline
  \multirow{ 5}{*}{\rotatebox[origin=l]{90}{Case 1}} 
  &Length & 2.0 & 1.0 & 2.0 & 1.0 & 2.0 \\ \cline{2-7}
  &$\sigma_s$ & 0.0 & 0.0  & 0.0 & 0.9 & 0.9 \\ \cline{2-7}
  &$\sigma_t$ &50.0& 5.0 & 0.0 & 1.0 &1.0  \\ \cline{2-7}
  &$q$ & 50. & 0.0 & 0.0 & 1.0 & 1.0 \\ \cline{2-7}
  &\#dofs & 25 & 25 & 25 & 25 & 25 \\  \cline{2-7}
  & B.C.& \multicolumn{5}{|c|}{Vac.}   \\\hline
\end{tabular}
\begin{tabular}{|r|c|c|} \hline
  & \#zones & \multicolumn{1}{|c|}{1}  \\ \hline
  \multirow{6}{*}{\rotatebox[origin=l]{90}{Case 2}} 
  &Length & 10.0  \\ \cline{2-3}
  &$\sigma_s$ & 100.0 \\ \cline{2-3}
  &$\sigma_t$ &100.0  \\ \cline{2-3}
  &$q$ & 0.0  \\ \cline{2-3}
  &\#dofs & 100 \\ \cline{2-3}
  & B.C.& $\psi_{5}(0)=0$  \\ \hline
\end{tabular}
\begin{tabular}{|r|c|c|} \hline
  & \#zones & \multicolumn{1}{|c|}{1}  \\ \hline
  \multirow{6}{*}{\rotatebox[origin=l]{90}{Case 3}} 
  &Length & 10.0  \\ \cline{2-3}
  &$\sigma_s$ & 10.0 \\ \cline{2-3}
  &$\sigma_t$ &10.0  \\ \cline{2-3}
  &$q$ & 0.1  \\ \cline{2-3}
  &\#dofs & 100 \\ \cline{2-3}
  & B.C.& Vac. \\ \hline
\end{tabular}
\begin{tabular}{|r|c|c|} \hline
  & \#zones & \multicolumn{1}{|c|}{1}  \\ \hline
  \multirow{6}{*}{\rotatebox[origin=l]{90}{Case 4}} 
  &Length & 100.0  \\ \cline{2-3}
  &$\sigma_s$ & 0.09999 \\ \cline{2-3}
  &$\sigma_t$ &0.1  \\ \cline{2-3}
  &$q$ & 1.0  \\ \cline{2-3}
  &\#dofs & 100 \\ \cline{2-3}
  & B.C.& Vac. \\ \hline
  \end{tabular}%
\caption{Data for the one-dimensional test cases.} \label{Tab:oneD_data}%
\end{table}

The results are reported in Figure~\ref{Fig:oneD_results}. We show in
Panels~\eqref{Fig:oneD_results:a}-\eqref{Fig:oneD_results:c} the
scalar flux for the dG1 approximation (labeled dG1) and for the positive asymptotic
preserving technique (labeled AP cG1).  We observe a fair agreement between the two
methods given the number of grid points.
Panel~\eqref{Fig:oneD_results:d} shows the angular flux $\psi_{h,1}$
for case 4.  For this case the dG1 approximation gives negatives
values at $x=100$ on the angular fluxes $1$, $2$, and $3$ (the values
are  $-0.24$, $-0.22$, $-0.066$, respectively (approximated to 2 digits)).
In all the cases the asymptotic
preserving technique is always nonnegative.
   
\begin{figure}[h]
\begin{subfigure}[t]{0.25\textwidth} 
\includegraphics[width=\textwidth]{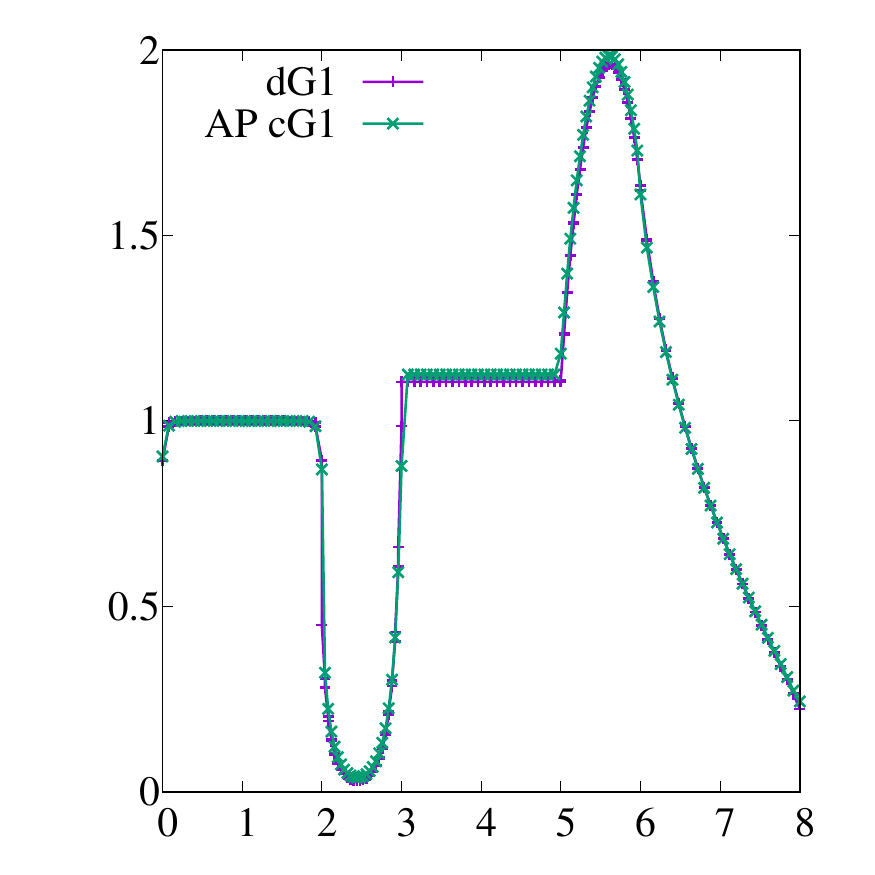}%
\caption{Case 1, $\opsi_h$}\label{Fig:oneD_results:a}
\end{subfigure}%
\begin{subfigure}[t]{0.25\textwidth}
\includegraphics[width=\textwidth]{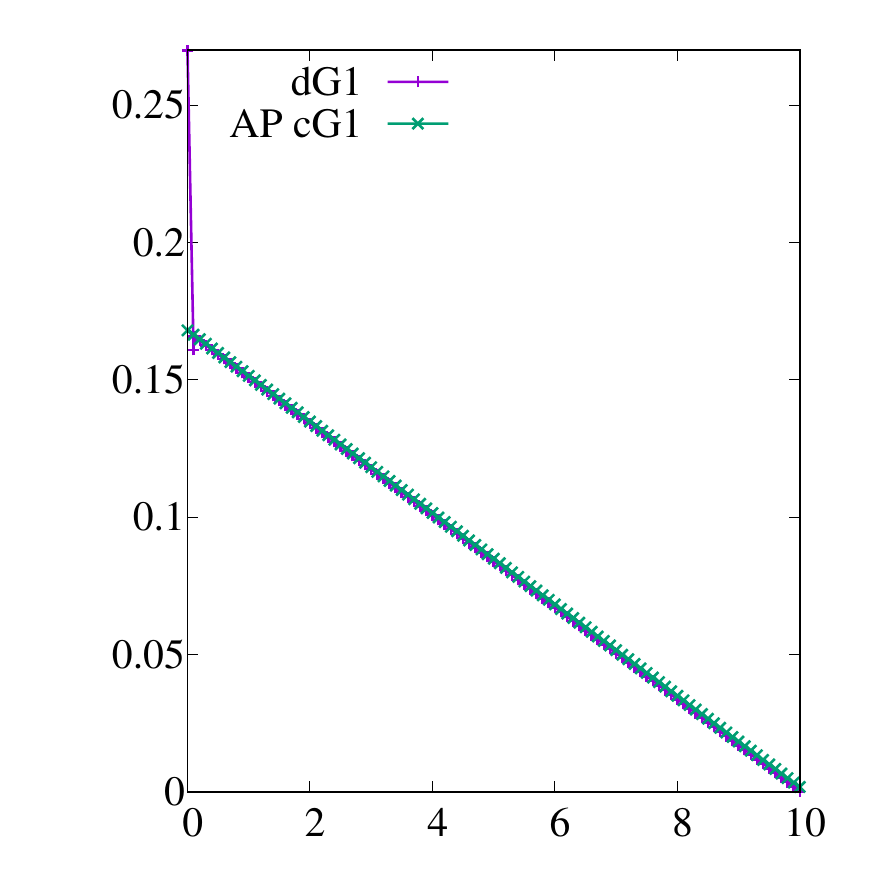}%
\caption{Case 2, $\opsi_h$}\label{Fig:oneD_results:b}
\end{subfigure}%
\begin{subfigure}[t]{0.25\textwidth}
\includegraphics[width=\textwidth]{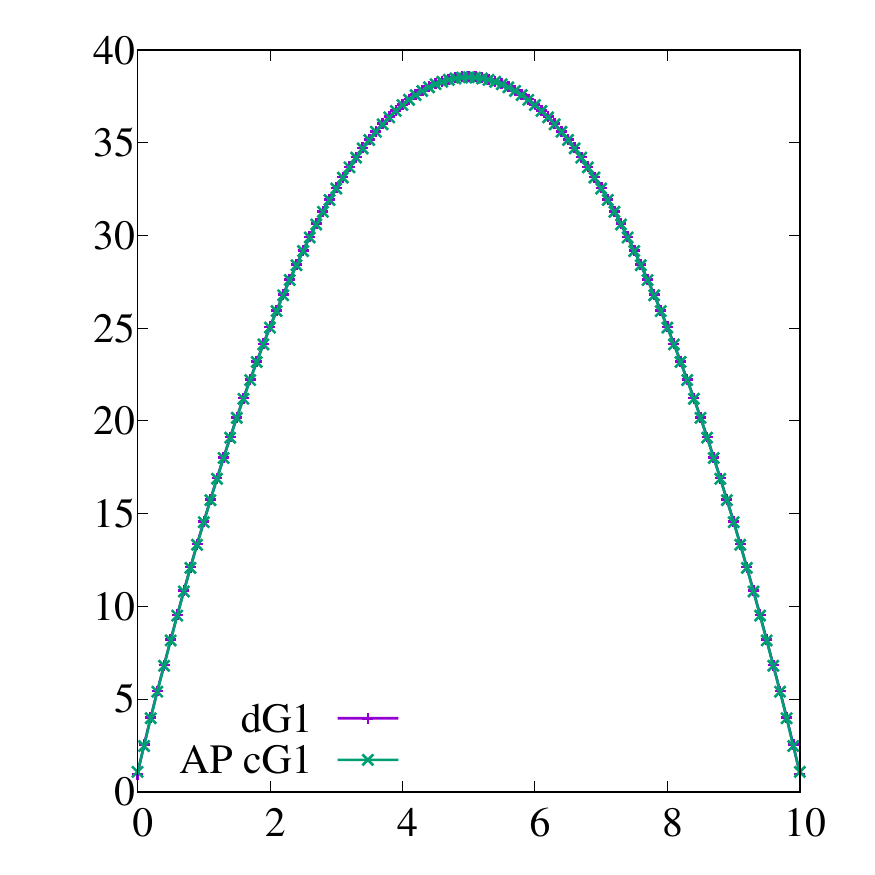}%
\caption{Case 3, $\opsi_h$}\label{Fig:oneD_results:c}
\end{subfigure}%
\begin{subfigure}[t]{0.25\textwidth}
\includegraphics[width=\textwidth]{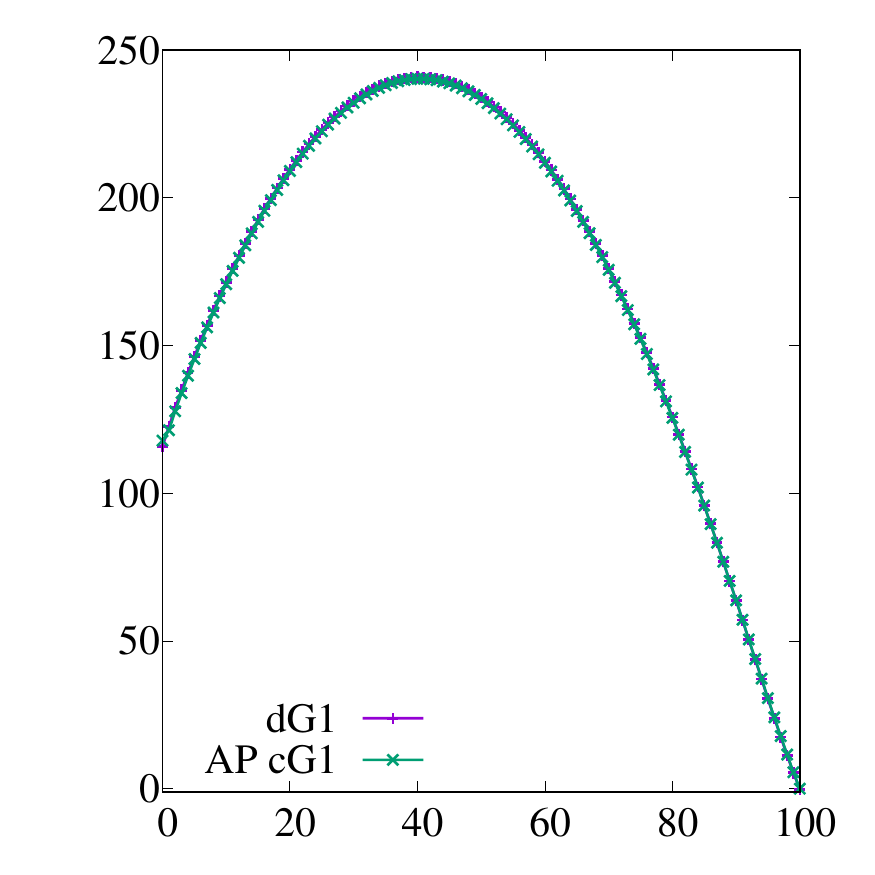}%
\caption{Case 4, $\psi_{h,1}$}\label{Fig:oneD_results:d}
\end{subfigure}\vspace{-\baselineskip}
\caption{Comparison between the (first-order) positive, asymptotic preserving
 cG1 method and the (second-order) upwind dG1 method.}\vspace{-\baselineskip}  \label{Fig:oneD_results}
\end{figure}

\subsection{Boundary effects}
We consider the problem~\eqref{model_pb} in the two-dimensional domain
$\Dom=(0,100)^2\CROSS \Real$ with uniform cross sections
$\sigma_t(\bx)=0.1$, $\sigma_s(\bx)=0.0999$ and uniform source term
$q(\bx,\bOmega)=1$ for all $(\bx,\bOmega)\in\Dom\CROSS \calS$.  The
boundary condition is set to zero $\alpha(\bx,\bOmega)=0$ for all
$(\bx,\bOmega)\in\front_-$.  (This is the two-dimensional counterpart
of the one-dimensional case 4 discussed in \S\ref{Sec:oneD}.)  We use
the $S_6$  quadrature for the discrete
ordinates (24 directions in 2D). The approximation in space for the
asymptotic preserving method is done on a non-uniform grid composed of
$151294$ triangles with $76160$ grid points (\ie $1\,829\,520$ dofs in
total).  The dG1 approximation is done with $64\CROSS 64$ cells, that is
$16384$ space dofs (\ie $393\,216$ dofs in total).

\begin{figure}[h]  \centering
\begin{subfigure}[t]{0.25\textwidth}
\includegraphics[width=\textwidth]{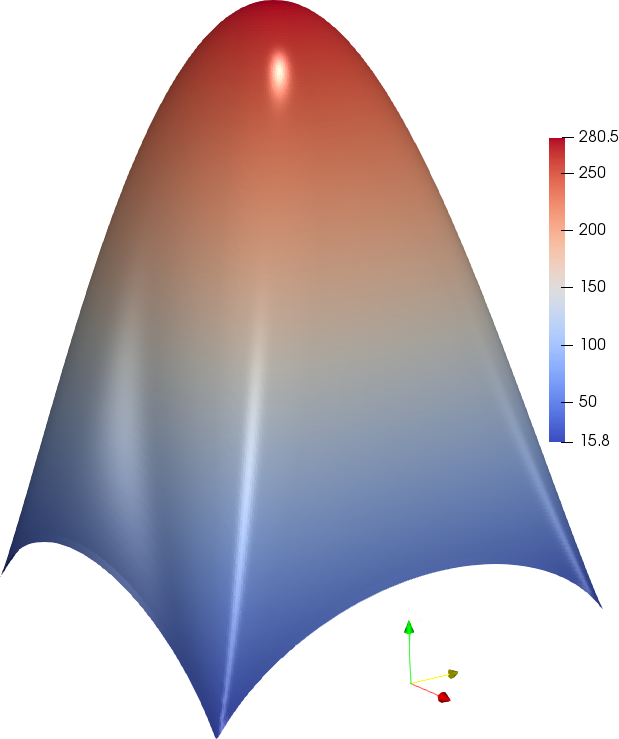}%
\caption{$\opsi_h$, AP scheme} 
\end{subfigure}%
\begin{subfigure}[t]{0.25\textwidth}
\includegraphics[width=\textwidth]{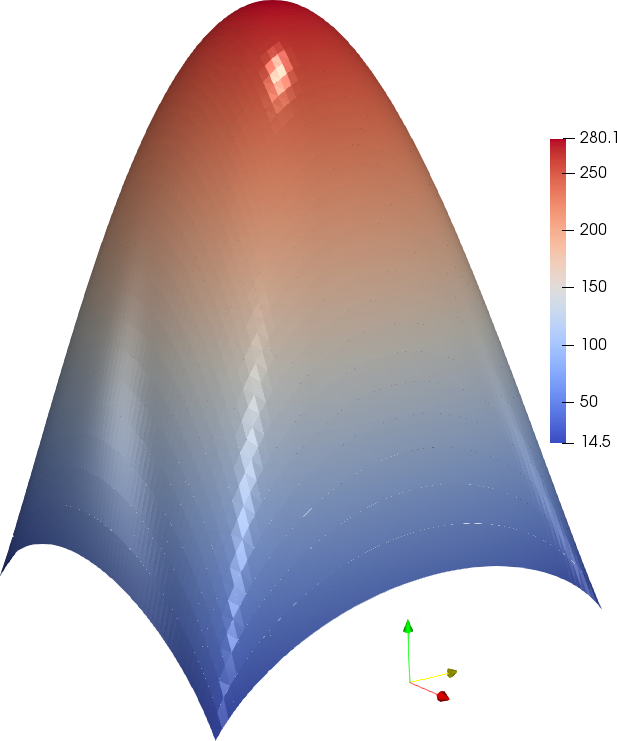}%
\caption{$\opsi_h$, dG1 scheme}
\end{subfigure}%
\begin{subfigure}[t]{0.25\textwidth}
\includegraphics[width=\textwidth]{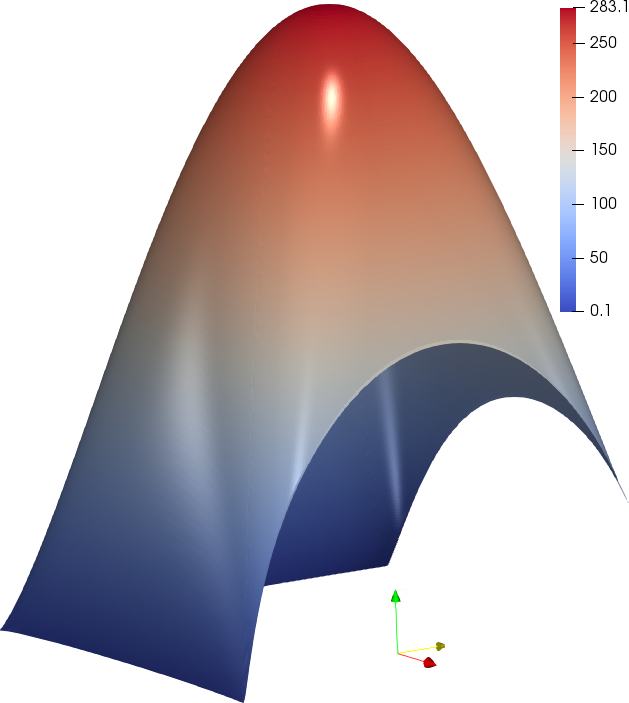}
\caption{$\psi_{h,1}$, AP scheme}
\end{subfigure}%
\begin{subfigure}[t]{0.25\textwidth}
\includegraphics[width=\textwidth]{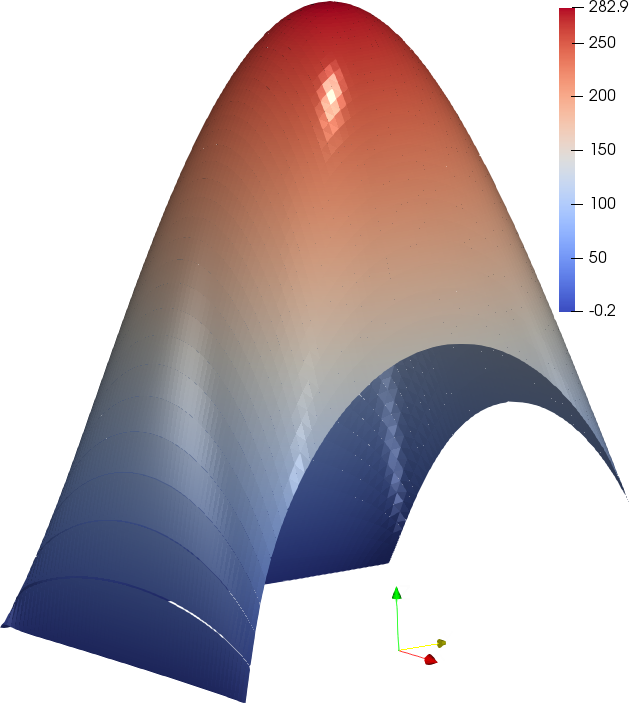}
\caption{$\psi_{h,1}$, dG1 scheme} \label{Fig:2D_BC_effects:d}
\end{subfigure}\vspace{-\baselineskip}
\caption{Scalar $\opsi_h$ and angular flux $\psi_{h,1}$.}\vspace{-\baselineskip}
\label{Fig:2D_BC_effects}
\end{figure}
We show in Figure~\ref{Fig:2D_BC_effects} the scalar flux and the
angular flux corresponding to the first angle $\bOmega_1$. We have
verified that the angular fluxes for the asymptotic preserving method
are all non-negative, as expected, but the upwind dG1 approximation
gives negative angular fluxes. In particular, we observe in
Panel~\eqref{Fig:2D_BC_effects:d} that the minimum value of the first
angular flux of the dG1 approximation is equal to $-0.2$ (1 digit
roundoff approximation.)  


\subsection{Reflection effects}
We now consider the two-dimensional problem with reflection effects.
The domain is $\Dom=(0,1)^2\CROSS \Real$ with uniform cross sections
$\sigma_t(\bx)=100$, $\sigma_s(\bx)=99$ if $x_2\ge 0.5$ (optically
thick and diffusive zone), and $\sigma_t(\bx)=\sigma_s(\bx)=0$ if
$x_2\le 0.5$ (void).  We use the $S_6$ 
quadrature. The left boundary is illuminated with intensity $1$ along
the first direction of the quadrature
$\bOmega_1:=(0.93802334,0.25134260,0.23861919)$ (eight digits
truncation). The incoming flux is set to $0$ along the bottom boundary
for $\bOmega_1$.  For all the other angular fluxes we set
$\psi_{h,k|\front_-}=0$, $k\in\calL{\setminus}\{1\}$.  The
approximation in space for the asymptotic preserving method is done on
a non-uniform grid composed of $151434$ triangles with $76230$ grid
points (\ie $1\,829\,520$ dofs in total). The dG1 computation is done
with $256\CROSS 256$ cells to ascertain the accuracy of the
  solution since it is our reference; that makes $262144$ dofs for the space
approximations (\ie $6\,291\,456$ dofs in total). 

\begin{figure}[h]  \centering
\begin{subfigure}[t]{0.25\textwidth}
\includegraphics[width=\textwidth]{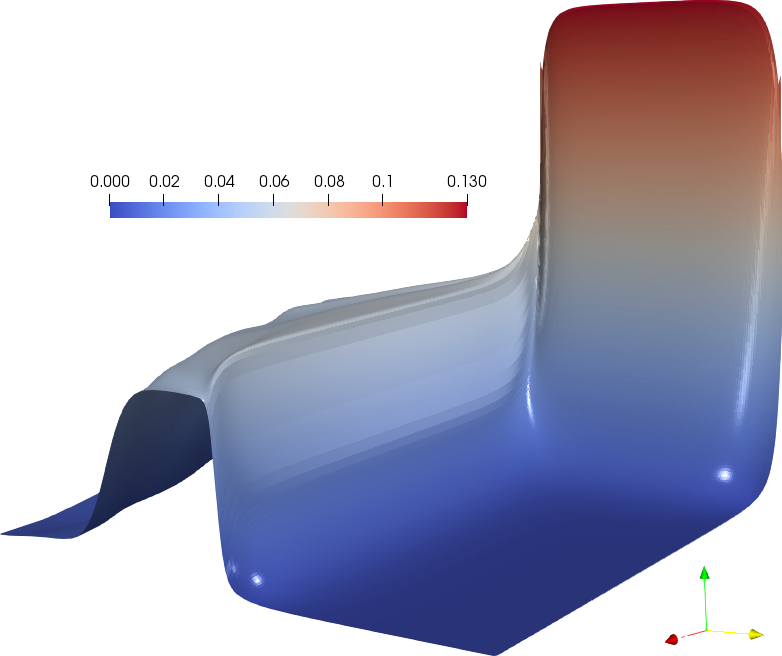}%
\caption{$\opsi_h$, AP scheme}
\end{subfigure}%
\begin{subfigure}[t]{0.25\textwidth}
\includegraphics[width=\textwidth]{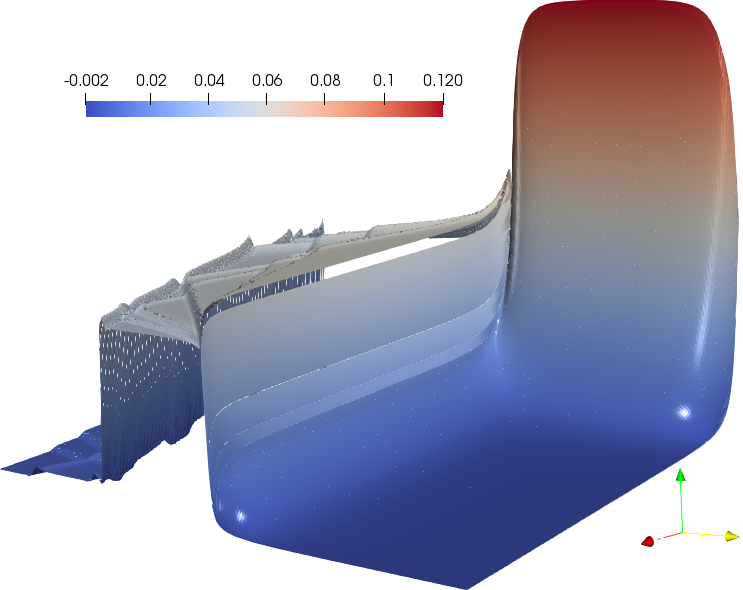}%
\caption{$\opsi_h$, dG1 scheme} \label{Fig:2D_two_dom:b}
\end{subfigure}%
\begin{subfigure}[t]{0.25\textwidth}
\includegraphics[width=\textwidth]{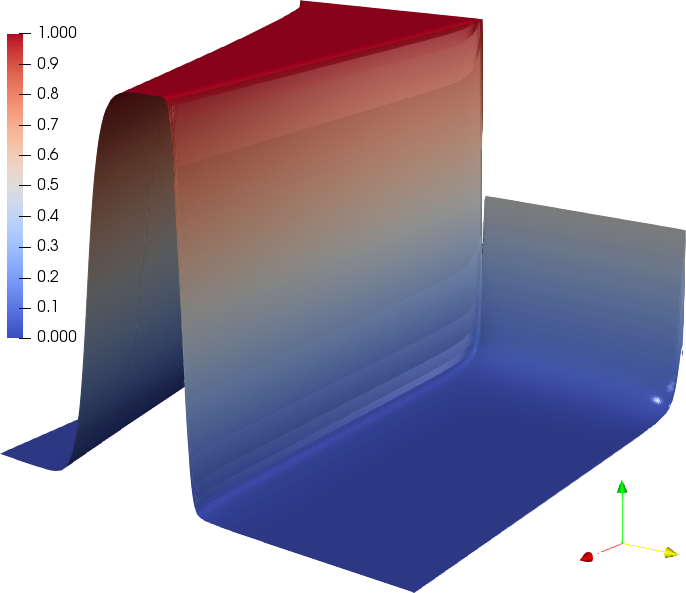}
\caption{$\psi_{h,1}$, AP scheme}
\end{subfigure}%
\begin{subfigure}[t]{0.25\textwidth}
\includegraphics[width=\textwidth]{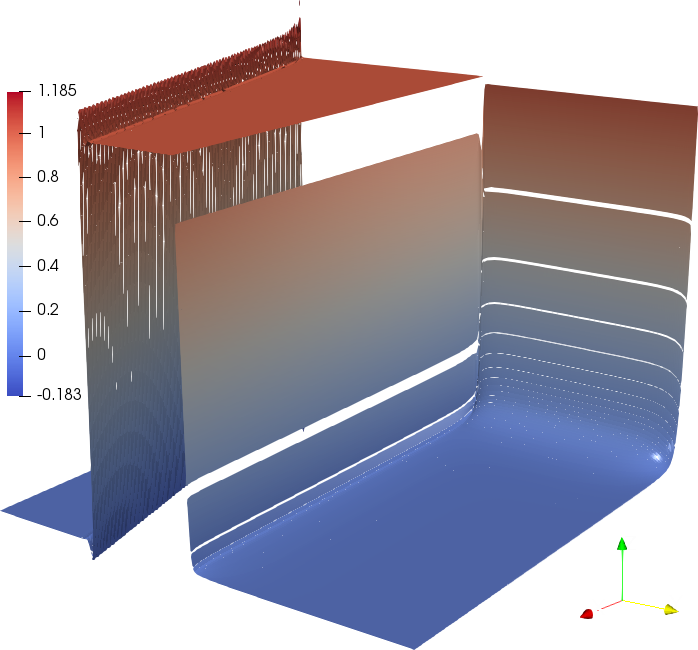}
\caption{$\psi_{h,1}$, dG1 scheme} \label{Fig:2D_two_dom:d}
\end{subfigure}\vspace{-\baselineskip}
\caption{Scalar $\opsi_h$ and angular flux $\psi_{h,1}$.}\vspace{-\baselineskip}
\label{Fig:2D_two_dom}
\end{figure}
We show in Figure~\ref{Fig:2D_two_dom} the scalar flux and the angular
flux corresponding to the first angle $\bOmega_1$. The angular fluxes
for the asymptotic preserving method are all non-negative, but the
upwind dG1 approximation gives negative values for the scalar flux and
the angular fluxes.  We observe that the minimum value of the dG1
scalar flux is approximately $-0.002$
(Panel~\eqref{Fig:2D_two_dom:b}), and the minimum value is $-0.183$
for the first angular flux (Panel~\eqref{Fig:2D_two_dom:d}). The dG1
approximation is obviously more accurate than the asymptotic preserving solution,
  but it experiences overshoots and undershoots at the interfaces
between the two materials, whereas the positive asymptotic preserving
solution does not.

\section{Conclusions} \label{Sec:Conclusions} 
We have introduced a (linear) positive asymptotic preserving method
for the approximation of the one-group radiation transport equation
(see~\eqref{AP_formulation_cG_plus_BC}). The approximation in space is
discretization agnostic: the approximation can be done with continuous
or discontinuous finite elements (or finite volumes).  The method is
first-order accurate in space. This type of accuracy is coherent with
Godunov's theorem since the method is linear.  The two key theoretical
results of the paper are Theorem~\ref{Th:diffusion_limit_AP} and
Theorem~\ref{Th:AP_positivity}. We have illustrated the performance of the method with
continuous finite elements. We have observed that the method converges
with the rate $\calO(h)$ in the $L^2$-norm on manufactured solutions.
It converges with the rate $\calO(h^2)$ in the $L^2$-norm in the
diffusion limit.  The method has also been observed to be non-negative
(in compliance with Theorem~\ref{Th:AP_positivity}).  It does not
suffer from overshoots like the upwind dG1 approximation at the
interfaces of optically thin and optically thick regions.

The present work is the first part of a ongoing project aiming at
developing techniques that are high-order accurate,
positivity-preserving, and robust in the diffusion limit.  To reach
higher-order accuracy the technique must be made nonlinear.  This
could be done by using smoothness indicators like in
\citep[\S4.3]{Guermond_Popov_2017}, or by using limiting technique, or
by enforcing positivity through inequality constraints (see \eg
\cite[\S4]{Hauck_McClarren_2010}).  Our progresses in this direction
will be reported elsewhere.

\bibliographystyle{abbrvnat} 
\bibliography{ref}

\end{document}